\documentclass[10pt,a4paper]{article}

\usepackage{fullpage,amssymb,amsmath,amsfonts,hyperref,amsthm,paralist,graphicx,color,float, mathtools,tikz}
\usepackage{here}
\usepackage[normalem]{ulem}
\usepackage{centernot}
\usepackage{ifthen,xkeyval, tikz, calc, graphicx}
\usepackage{xcolor}
\usepackage{caption, subcaption}
\usepackage[textsize=scriptsize]{todonotes}
\usepackage{dsfont}
\usepackage{mathrsfs}
\usepackage{comment}

\newtheorem{theorem}{Theorem}[section]
\newtheorem{lemma}[theorem]{Lemma}
\newtheorem{prop}[theorem]{Proposition}

\newtheorem{observation}[theorem]{Observation}

\theoremstyle{definition}
\newtheorem{definition}[theorem]{Definition}

\newcommand{\al}[1]{\begin{align*}#1\end{align*}}
\newcommand{\aln}[1]{\begin{align}\begin{split}#1\end{split}\end{align}}
\newcommand{\algn}[1]{\begin{align}#1\end{align}}
\newcommand{\eqq}[1]{\begin{equation}#1\end{equation}}

\newcommand{\p}{\mathbb P}
\newcommand{\pp}{\mathbb P_p}
\newcommand{\E}{\mathbb E}
\newcommand{\R}{\mathbb R}
\newcommand{\Rd}{\mathbb R^d}

\newcommand{\Zd}{\mathbb Z^d}

\newcommand{\N}{\mathbb N}

\newcommand{\C}{\mathscr {C}}
\newcommand{\thinn}[1]{\langle #1 \rangle}
\newcommand{\piv}[1]{\textsf {Piv}(#1)}
\newcommand{\taup}{\tau_p}
\newcommand{\taupf}{\tau_p^\bullet}
\newcommand{\taupo}{\tau_p^\circ}

\newcommand{\conn}[3]{#1 \longleftrightarrow #2\textrm { in } #3}
\newcommand{\nconn}[3]{#1 \centernot\longleftrightarrow #2\textrm { in } #3}

\newcommand{\jeq}{J}

\newcommand{\throughconn}[1]{\xleftrightarrow{\ #1 \ }}

\newcommand{\lconn}[1]{\xleftrightarrow{\,\, (#1) \,\,}}
\newcommand{\lcyc}[1]{\xLeftrightarrow{\,\, (#1) \,\,}}

\newcommand{\ftau}{\widehat\tau_p}

\newcommand{\connf}{D}

\newcommand{\fconnf}{\widehat\connf}

\newcommand{\fspace}{(-\pi,\pi]^d}
\newcommand{\fpip}{\widehat\Pi_p}

\newcommand{\trip}{\triangle_p}

\newcommand{\tripf}{\triangle_p^{\bullet}}
\newcommand{\tripof}{\triangle_p^{\bullet\circ}}
\newcommand{\tripoff}{\triangle_p^{\bullet\bullet\circ}}
\newcommand{\dtri}{\triangleleft \hspace{-1 pt} \triangleright}
\newcommand{\dtrial}{\triangleleft \hspace{-3.2 pt} \triangleright}
\newcommand{\e}{\text{e}}
\newcommand{\orig}{\mathbf{0}}
\renewcommand{\i}{\text{i}}

\renewcommand{\O}{\mathcal O}
\newcommand{\omtwo}{\O(\Omega^{-2})}
\newcommand{\omone}{\O(\Omega^{-1})}
\newcommand{\gen}[1]{\langle\!\langle #1 \rangle\!\rangle}

\numberwithin{equation}{section}
\allowdisplaybreaks

\definecolor{darkorange}{RGB}{255,165,0}
\definecolor{altviolet}{RGB}{139,0,139}
\definecolor{turquoise}{RGB}{64,224,208}

\title{Expansion for the critical point of site percolation: \\the first three terms}
\author{Markus Heydenreich\footnote{Ludwig-Maximilians-Universit\"at M\"unchen, Mathematisches Institut, Theresienstr.\ 39, 80333 M\"unchen, Germany. \newline E-mail: m.heydenreich@lmu.de; kilianmatzke@web.de} 
		\and Kilian Matzke\footnotemark[1] }

\begin{document}
\maketitle

\begin{abstract}
We expand the critical point for site percolation on the $d$-dimensional hypercubic lattice in terms of inverse powers of $2d$, and we obtain the first three terms rigorously. This is achieved using the lace expansion. 
\end{abstract}



\section{Introduction} \label{sec:introduction}

We study site percolation on the hypercubic lattice $\Zd$. To this end, we fix a parameter $p\in[0,1]$ and create a random subgraph of $\Zd$ as follows. Each site (or vertex) $x\in\Zd$, independently of all other sites, is declared \emph{occupied} with probability $p$ (and \emph{vacant} otherwise). A bond (edge) between two nearest-neighbor sites in $\Zd$ is an edge of the random subgraph if and only if the two sites are occupied. Denote by $\theta(p)$ the probability that there is a path starting at the origin $\orig\in\Zd$ and diverging to infinity that consists only of occupied vertices. This allows us to define the \emph{critical point} as 
	\eqq{ p_c := \inf\big\{p\in[0,1]: \theta(p)>0\big\}. \label{eq:intro:p_c_def}}
It is standard that $0<p_c<1$ in all dimensions $d\ge2$. In general, it is not possible to write down an explicit value for $p_c=p_c(d)$ (see Table \ref{table-pcnum} for numerical values), a notable exception is site percolation on the two-dimensional triangular lattice (when $p_c=1/2$). However, it is possible to derive an asymptotic expansion for $p_c(d)$ when $d\to\infty$. Indeed, it is known in the physics literature that
	\eqq{ p_c = \sigma^{-1} + \frac 32 \sigma^{-2} + \frac{15}{4} \sigma^{-3} + \frac{83}{4}\sigma^{-4} + \frac{6577}{48}\sigma^{-5}+ \frac{119077}{96}\sigma^{-6}+\cdots 
					\quad\text{for $\sigma={2d-1}\to\infty$.} \label{eq:intro:expansion-phys}}
The first four terms were found by Gaunt, Ruskin, and Sykes in 1976~\cite{GauRusSyk76} through exact enumeration, the last two terms have been obtained by Mertens and Moore~\cite{MerMoo18} by exploiting involved numerical methods. When writing this in powers of $\frac{1}{2d}$,~\eqref{eq:intro:expansion-phys} becomes
	\eqq{	p_c(d) = (2d)^{-1} + \frac 52 (2d)^{-2} + \frac{31}{4}(2d)^{-3} + \frac{75}{2} (2d)^{-4} + \frac{11977}{48}(2d)^{-5}+ \frac{209183}{96}(2d)^{-6}+\cdots. \label{eq:intro:expansion-phys2} }
In this paper, we extend the previously known first term by establishing the second and third term, including a rigorous bound on the error term.
\begin{theorem}[Expansion of $p_c$ in terms of $(2d)^{-1}$] \label{thm:expansion_of_p_c}
As $d \to \infty$,
	\[p_c(d) = (2d)^{-1} + \frac 52 (2d)^{-2} + \frac{31}{4} (2d)^{-3} + \mathcal O\left( (2d)^{-4} \right). \]
\end{theorem}

The key technical tool for our approach is the lace expansion for site percolation. It was established in a recent paper~\cite{HeyMat20}, which itself draws its inspiration from Hara and Slade's seminal paper~\cite{HarSla90}. 
The lace expansion provides an expression for $p_c$ in terms of \emph{lace-expansion coefficients}, which are defined in Definition~\ref{def:prelim:lace_expansion_coefficients}. Moreover, it provides good control over these coefficients, and the results of \cite{HeyMat20} identify already the leading order term in~\eqref{eq:intro:expansion-phys2}. 

\paragraph{Comparison with bond percolation.}
It is most instructive to compare the critical thresholds for site and bond percolation. 
While the critical behaviour of bond- and site percolation is comparable, the actual values of the critical thresholds differ, as illustrated by the following table: 
\begin{table}[h]\centering\small
\begin{tabular}{l|c|c|c|c|c|c|c|c|c|c|c}
{dim} & 2 & 3 & 4 & 5 & 6 & 7 & 8 & 9 & 10 & 11 & 12 \\ \hline
$p_c^{\text{site}}$ & 0.5927 & 0.3116 & 0.1969 & 0.1408 & 0.1090 & 0.0890 & 0.0752 & 0.0652 & 0.0576 & 0.0516 & 0.0467   \\ \hline 
$p_c^{\text{bond}}$ & $0.5 ^\ast$ & 0.2488 & 0.1601 & 0.1182 & 0.0942 & 0.0786 & 0.0677 & 0.0595 & 0.0531 & 0.0479 &  0.0437
\end{tabular}
\caption{Critical values for percolation on $\Zd$, rounded to multiples of $10^{-4}$. The only rigorously obtained value is for bond percolation in dimension 2 (marked with ${}^\ast$). All other values are obtained through numerical simulation; the values for $d\ge4$ are reported from Grassberger~\cite{Gra03} and Mertens and Moore~\cite{MerMoo18}.}
\label{table-pcnum}
\end{table}

Grimmett and Stacey~\cite{GriSta98} prove that $p_c^{\text{site}}>p_c^{\text{bond}}$ on $\Zd$ for all dimensions $d\ge2$. This difference must be reflected in the asymptotic expansion for $p_c$. Indeed, Hara and Slade~\cite{HarSla95} and van der Hofstad and Slade~\cite{HofSla06} rigorously obtain a series expansion for \emph{bond} percolation as 
\begin{equation}
	p_c^{\text{bond}}(d) = (2d)^{-1} + (2d)^{-2} + \frac{7}{2}(2d)^{-3} + \mathcal O\left((2d)^{-4}\right), \label{eq:intro:pc_bond_expansion}
\end{equation}
which indeed differs from the expansion of $p_c^{\text{site}}$ in Theorem \ref{thm:expansion_of_p_c}. 
Again, more precise estimates are known by non-rigorous methods~\cite{GauRus78,MerMoo18}: 
	\eqq{ p_c^{\text{bond}} =  \sigma^{-1} + \frac 52 \sigma^{-3} + \frac{15}{2}\sigma^{-4} + 57\sigma^{-5}+ \frac{4855}{12}\sigma^{-6}+\cdots \label{eq:expansion-bond-phys} }
for $\sigma={2d-1}$, which is equivalent to 
	\[ p_c^{\text{bond}}(d) = (2d)^{-1} + (2d)^{-2} + \frac{7}{2}(2d)^{-3} + 16 (2d)^{-4} + 103 (2d)^{-5} + \frac{9487}{12}(2d)^{-6} + \cdots .\]
We remark that~\eqref{eq:intro:pc_bond_expansion} was proved in~\cite{HofSla06} also for the $d$-dimensional cube. More recently, an asymptotic expansion was also proven for the Hamming graph~\cite{FedHofHolHul20}.

\paragraph{Borel summability of the coefficients.} 
Theorem \ref{thm:expansion_of_p_c} establishes an expansion to the third order, but it is plausible that even an expansion to \emph{all} orders for site percolation exist: writing $s=\frac1{2d}$ and $\bar p_c(s)=p_c(d)$, this means that there is a real sequence $(\alpha_n)_{n\in\N}$ such that for any $M\in\N$, 
\begin{equation}\label{eqPcExpansion}
	\bar p_c(s)=\sum_{n=1}^{M-1}\alpha_n\,s^n+\mathcal O(s^M).
\end{equation} 
The corresponding statement for bond percolation was proved by Hofstad and Slade~\cite{HofSla05}. 
However, it is expected that the radius of convergence of the series $\sum\alpha_ns^n$ is zero (even though rigorous evidence is lacking), and this non-convergence is valid in greater generality for series expansions of critical thresholds of various statistical mechanical models. The reason is that the sequence of absolute values $|\alpha_1|,|\alpha_2|,|\alpha_3|,\dots$ grows very rapidly (with sign changes for higher $n$), and that therefore it is not possible to compute $\bar p_c(s)$ from the sequence $(\alpha_n)$. 

Instead, we believe that the coefficients $(\alpha_n)$ are \emph{Borel summable}: suppose $\bar p_c(s)$ has an analytic extension to the complex disk $D=\{z\in\mathbb C\colon {\rm Re}(z^{-1})>1\}$, and suppose further that there is $C>0$ such that for all $s\in D$ and all $M$, we have 
\begin{equation}\label{eq-BorelTypeBound}
	\left|\bar p_c(s)-\sum_{n=1}^{M-1}\alpha_n\,s^n\right|\le C^M\,s^M\,M!,
\end{equation}
then Sokal~\cite{Sok80} proves that the \emph{Borel transform} $B(t)=\sum_{n=1}^\infty{\alpha_n t^n}/{n!}$ exists, and $\bar p_c(s)$ equals the \emph{Borel sum} 
\begin{equation}
	\bar p_c(s)=\frac1s\int_0^\infty e^{-t/s}B(t)\,dt.
\end{equation}
It is, however, unclear how an analytic extension of $\bar p_c(s)$ for site percolation could be obtained. 

A rare example for which we know Borel summability is the exact solution $K_c(d)$ of the spherical model. Gerber and Fisher~\cite{GerFis74} prove that there is an expansion of $K_c(d)$ in powers of $1/d$, that the radius of convergence is zero, but that we may interpret the expansion as a Borel sum as described above. 
They also prove that the signs of the coefficients of $K_n$ oscillate: the first 12 terms are positive, the next 8 are negative, the next 9 are positive, and so on. 
For the well-known model of self-avoiding walk, Graham~\cite{Gra10} proves bounds for the connective constant as in \eqref{eq-BorelTypeBound}.

\subsection{Strategy of proof, outline of the paper}
Theorem~\ref{thm:expansion_of_p_c} heavily builds upon the results obtained in~\cite{HeyMat20}. We use Section~\ref{sec:prelim} to collect the necessary notation and results from~\cite{HeyMat20} in order to prove our main result. At the heart of these results is an identity for $\taup$. From this, we almost immediately get an identity for $p_c$ in terms of so-called \emph{lace-expansion coefficients} (see Definition~\ref{def:prelim:lace_expansion_coefficients}). It will be clear that sufficient control over the coefficients will result in the expansion of Theorem~\ref{thm:expansion_of_p_c}. In fact, the results from~\cite{HeyMat20} immediately give the first term of~\eqref{eq:intro:expansion-phys2}.

For the other terms in Theorem~\ref{thm:expansion_of_p_c}, however, we require even better control of these coefficients, which is provided by Lemma~\ref{lem:exp:Pis_expansion}. Section~\ref{sec:expansion_critical_point} proves Theorem~\ref{thm:expansion_of_p_c} assuming Lemma~\ref{lem:exp:Pis_expansion}. The latter is at the heart of this paper and is proved in Section~\ref{sec:exp:bounds_on_Pis}. As a preparation for the proof, Section~\ref{sec:exp:bounds_l_connection} introduces some new notation on connection events and proves bounds on them. Those bounds are in essence an extension of the bounds presented in Section~\ref{sec:prelim}.

\section{Preliminaries} \label{sec:prelim}
\subsection{Site percolation: Model and basic definitions} \label{sec:prelim_intro}

We introduce the model more formally. Given $p \in [0,1]$, we can choose our probability space to be $(\{0,1\}^{\Zd}, \mathcal F, \pp)$, where the $\sigma$-algebra $\mathcal F$ is generated by the cylinder sets, and $\pp = \bigotimes_{x\in\Zd} \text{Ber}(p)$. We call $\omega\in\{0,1\}^{\Zd}$ a configuration and say that a site $x\in\Zd$ is \emph{open} or \emph{occupied} in $\omega$ if $\omega(x)=1$. If $\omega(x)=0$, we say that the site $x$ is \emph{closed} or \emph{vacant}. We often identify $\omega$ with the set $\{x\in\Zd: \omega(x)=1\}$.

For $k\in\N$ and a configuration $\omega$, we call $(v_0,v_1,\dots,v_k)\in(\Zd)^{k+1}$ an \emph{occupied path} of length $k$ from $v_0$ to $v_k$ if $| v_i-v_{i-1}| = 1$ for all $1 \leq i \leq k$, and $v_i\in\omega$ for $1 \leq i \leq k-1$. 
Here, and throughout the paper, we write $|x|= \sum_{i=1}^d|x_i|$ for $x\in\Rd$ (which is equal to the graph distance in $\Zd$). 
For two points $x \neq y\in\Zd$ we write $\{x \longleftrightarrow y\}$ (and say that $x$ is \emph{connected} to $y$) if there exists an occupied path from $x$ to $y$ of arbitrary length; mind that the this event is irrespective of the occupation status of $x$ and $y$. 
We set $\{x \longleftrightarrow x\}=\varnothing$, that is, $x$ is \emph{not} connected to itself. Moreover, $|x-y|=1$ implies $\{x \longleftrightarrow y\}=\{0,1\}^{\Zd}$ (neighbors are always connected).

We define the \emph{cluster} of $x$ to be $\C(x) = \{x\} \cup \{y \in \omega: x \longleftrightarrow y\}$. Note that apart form $x$ itself, points in $\C(x)$ need to be occupied.

The \emph{two-point function} $\tau_p\colon \Zd \to [0,1]$ is defined as $\tau_p(x):=\pp(\orig \longleftrightarrow x)$, where $\orig$ denotes the origin in $\Zd$. The \emph{percolation probability} is defined as $\theta(p) = \pp(\orig \longleftrightarrow \infty) = \pp( | \C(\orig)| = \infty)$. We note that $p \mapsto \theta(p)$ is increasing and define the \emph{critical point} for $\theta$ as in~\eqref{eq:intro:p_c_def}. The critical point $p_c$ depends on the underlying graph.

For an absolutely summable function $f\colon\Zd \to \R$, the discrete Fourier transform is defined as $\widehat f\colon\fspace \to \mathbb C$, where
	\[ \widehat f(k) = \sum_{x \in \Zd} \e^{\i k \cdot x} f(x)\]
and $k\cdot x = \sum_{j=1}^{d} k_j x_j$ denotes the scalar product.

\subsection{The lace expansion in high dimension}
We use this section to state the definitions and results from~\cite{HeyMat20} needed in the proof of Theorem~\ref{thm:expansion_of_p_c}. We note that the below definition uses the notion of disjoint occurrence (denoted $`\circ$') related to the BK inequality (which we will use at a later stage as well). For details on both, see e.g.~\cite[Chapter 2]{BolRio06} or \cite[Section 2.3]{Gri99}.

\begin{definition}[Connection events, modified clusters] \label{def:prelim:connection_events}
Let $x,u\in\Zd$ and $A \subseteq \Zd$.
\begin{enumerate}
\item We set $\Omega := 2d$.
\item We define $\jeq(x) := \mathds 1_{\{|x|=1\}} = \mathds 1_{\{0\sim x\}}$ and $\connf := J/\Omega$.
\item Let $\{\conn{u}{x}{A}\}$ be the event that there is a path from $u$ to $x$, all of whose internal vertices are elements of $\omega\cap A$.
\item We define $\{u \Longleftrightarrow x\} := \{u \longleftrightarrow x\} \circ \{u \longleftrightarrow x\} $ and say that $u$ and $x$ are \emph{doubly connected}.
\item We define the modified cluster of $x$ with a designated vertex $u$ as
		\[\widetilde\C^{u}(x) := \{x\} \cup \{y \in \omega \setminus\{u\} : x \longleftrightarrow y \text{ in } \Zd \setminus\{u\}  \} . \]
\item Let $\langle A \rangle := A \cup \{y \in \Zd: \exists x \in A: |x-y| = 1\}$.
\end{enumerate}
\end{definition}
Note that we introduce $\Omega=2d$. For better readability, we stick to using $\Omega$ for the remainder of the paper. We also address the Landau notation $f(\Omega) \leq \O(g(\Omega))$ that will appear frequently throughout the paper. It is always to be understood in the sense that there exists some $d_0$ and a constant $C(d_0)$, such that $f(\Omega) \leq C g(\Omega)$ for all $\Omega \geq d_0$. The constant $C$ may depend on other appearing parameters.

We remark that $\{\conn{x}{y}{\Zd}\} = \{x \longleftrightarrow y\} = \{\conn{x}{y}{\omega}\}$ and that $\{u \Longleftrightarrow x\} = \{0,1\}^{\Zd}$ for $|u-x| = 1$. Similarly, $\{u \Longleftrightarrow x\} = \varnothing$ for $u=x$. We state two elementary observations made in~\cite{HeyMat20} involving $\jeq$ that will be important later on.
\begin{observation}[{{Convolutions of $\jeq$, \cite[Observation 4.4]{HeyMat20}}}]\label{obs:prelim:J_convolutions}
Let $m\in\N$ and $x\in\Zd$ with $m \geq |x|$. Then there is a constant $c=c(m,x)$ with $c \leq m!$ such that
	\[ \jeq^{\ast m}(x) = c(m) \mathds 1_{\{m-|x| \text{ is even}\}} \Omega^{(m-|x|)/2 }.  \]
\end{observation}
\begin{observation}[{{Elementary bound on $\taup^{\ast n}$, \cite[Observation 4.5]{HeyMat20}}}] \label{obs:prelim:tau_J_extraction}
Let $m,n \in \N, p\in[0,1]$ and $x \in\Zd$. Then there is a constant $c=c(m,n)$ such that
	\[ \taup^{\ast n}(x) \leq c \sum_{l=0}^{m-1} p^{l} \jeq^{\ast l+n}(x) + c \sum_{j=1}^{n} p^{m+j-n}(\jeq^{\ast m} \ast \taup^{\ast j})(x). \]
\end{observation}
The following, more specific definitions are important to define the lace-expansion coefficients:
\begin{definition}[Extended connection events] \label{def:prelim:extended_connection_stuff} Let $v,u,x \in\Zd$ and $A \subseteq \Zd$.
\begin{enumerate}
\item Define
	\[ \{u \throughconn{A} x \} := \{u \longleftrightarrow x\} \cap \Big( \{\nconn{u}{x}{\Zd \setminus\thinn{A}}\} \cup  \{x \in \thinn{A}  \}  \Big). \]
In words, this is the event that $u$ is connected to $x$, but either any path from $u$ to $x$ has an interior vertex in $\thinn{A}$, or $x$ itself lies in $\thinn{A}$.
\item We introduce $\piv{u,x}$ as the set of pivotal points for $\{u \longleftrightarrow x\}$. That is, $v \in\piv{u,x}$ if the event $\{\conn{u}{x}{\omega \cup \{v\}}\}$ holds but $\{\conn{u}{x}{\omega\setminus\{v\}}\}$ does not.
\item Define the event
	\[ E'(v,u;A) := \{v \throughconn{A} u\} \cap \{ \nexists u' \in \piv{v,u}: v \throughconn{A} u'\} \]
\end{enumerate}
\end{definition}
We remark that $\{u \throughconn{\Zd} x\} = \{ u \longleftrightarrow x\}$. We can now define the lace-expansion coefficients. To this end, let $(\omega_i)_{i\in\N_0}$ be a sequence of independent site percolation configurations. For an event $E$ taking place on $\omega_i$, we highlight this by writing $E_i$. We also stress the dependence of random variables on the particular configuration they depend on. For example, we write $\C(u; \omega_i)$ to denote the cluster of $u$ in configuration $i$.
\begin{definition}[Lace-expansion coefficients] \label{def:prelim:lace_expansion_coefficients}
Let $n\in\N_0, x\in\Zd$, and $p \in [0,p_c]$. We define
	\al{ \Pi_p^{(0)}(x) &:= \pp(\orig \Longleftrightarrow x) - \jeq(x), \\
			\Pi_p^{(n)}(x) &:= p^n \sum_{u_0, \ldots, u_{n-1}} \pp \Big( \{\orig \Longleftrightarrow u_0\}_0 \cap \bigcap_{i=1}^{n} E'(u_{i-1},u_i; \C_{i-1})_i \Big), }
where $u_{-1}=\orig, u_n=x$ and $\C_{i} = \widetilde\C^{u_i}(u_{i-1};\omega_i)$. Let furthermore $\Pi_{p}(x) := \sum_{n=0}^{\infty} (-1)^n \Pi_p^{(n)}(x)$.
\end{definition}
It is proved in~\cite{HeyMat20} that the functions $(\Pi_p^{(n)}(x))_{n\in\N_0}$ are (absolutely) summable for every $x$ and that $\Pi_p$ is thus well defined. We remark that $E'(u_{i-1},u_i; \C_{i-1})_i$ takes place solely on $\omega_i$ only if $\C_{i-1}$ is regarded as a fixed set; otherwise it takes place on $\omega_{i-1}$ as well as $\omega_i$. Proposition~\ref{thm:prelim:convergence_of_LE} summarizes the main results of~\cite{HeyMat20} (namely, Theorem 1.1 and Proposition 4.2).

\begin{prop}[OZE, infra-red bound and bounds on the lace-expansion coefficients] \label{thm:prelim:convergence_of_LE} \ 
Let $p \in [0,p_c]$. Then there is $d_0 \geq 6$ such that, for all $d > d_0$, $\taup$ satisfies the Ornstein-Zernike equation
	\eqq{ \taup(x) = \jeq(x) + \Pi_p(x) + p\big((\jeq+\Pi_p)\ast\taup\big)(x). \label{eq:prelim:OZE}}
Secondly, there is a constant $C=C(d_0)$ such that
	\eqq{ p|\ftau(k)| \leq \frac{|\fconnf(k)| + C/d}{1-\fconnf(k)},  \label{eq:prelim:nfrared}}
where we take the right-hand side to be $\infty$ for $k=0$. Thirdly, $2dp \leq 1+C/d$, and lastly, for $n\in\N_0$,
	\eqq{ p\sum_{x\in\Zd} \Pi_{p}^{(n)}(x) \leq C (C/d)^{n \vee 1}.  \label{eq:prelim:Pi_(n)_bounds} }
\end{prop}
As a consequence, we also have $p \sum_{x} \Pi_p(x) \leq C/d$.

\subsection{Diagrammatic bounds}
In the proofs to follow, we need another result from~\cite{HeyMat20}. We formulate it in terms of a diagrammatic notation, as we are going to make use of this later as well. To this end, we introduce some quantities related to $\taup$.

\begin{definition}[Modified two-point functions and triangles] \label{def:prelim:modified_two-point_functions}
Let $x\in\Zd$ and define
	\[ \taupo(x) := \delta_{\orig,x} + \taup(x), \qquad \taupf(x) = \delta_{\orig,x} + p\taup(x).\]
Moreover, let  $\trip(x) = p^2(\taup\ast\taup\ast\taup)(x)$, $\tripf(x) = p(\taupf\ast\taup\ast\taup)(x)$, $\tripof(x) = p(\taupf\ast\taupo\ast\taup)(x)$, and $\tripoff(x) = (\taupf\ast\taupf\ast\taupo)(x)$. We also set
	\[ \trip = \sup_{x \in\Zd} \trip(x), \quad \tripf = \sup_{\orig \neq x \in\Zd} \tripf(x), \quad\tripof = \sup_{\orig \neq x \in\Zd} \tripof(x), \quad \tripoff = \sup_{x\in\Zd} \tripoff(x).  \]
\end{definition}

We need the following bounds obtained in~\cite{HeyMat20}.
\begin{prop}[{{Triangle bounds, \cite[Lemma 4.7]{HeyMat20}}}] \label{thm:prelim:triangle_bounds}
Let $p\in[0,p_c]$. Then there is $d_0\geq 6$ and a constant $C=C(d_0)$ such that, for all $d>d_0$,
	\[ \max\{ \trip, \tripf, \tripof\} \leq C /d, \qquad \max\{\tripf(\orig), \tripof(\orig),  \tripoff\} \leq C.  \]
\end{prop}

As part of the proof that bounds the functions $\Pi_p^{(i)}$ in~\cite{HeyMat20}, a first bound is formulated in terms of a long sum over products of the modified two-point functions. In a second step, those are decomposed into products of the modified triangles. We need a formulation of this intermediate bound on $\Pi_p^{(i)}$ for $i\in\{1,2\}$ for Section~\ref{sec:exp:bounds_on_Pis}, as well as a pictorial representation. We first state the needed bound on $\Pi_p^{(1)}$.

\begin{lemma}[{{Diagrammatic bound on $\Pi_p^{(1)}$, \cite[Lemma 3.10]{HeyMat20}}}] \label{lem:prelim:diagrammatic_bound_Pi1}
Let $p \in [0, p_c]$. Then
	\eqq{ \sum_{x\in\Zd} \Pi_p^{(1)}(x) \leq \sum_{\substack{w,u,t,z,x \in\Zd: \\ u \neq x, |\{t,z,x\}| \neq 2}} \taupf(w) \taup(u)\taup(w-u) 
					\taupo(z-w) \taupf(t-u) \taupf(z-t) \taupf(x-t) \taupo(x-z).  \label{eq:prelim:Pi1_diagrammatic_bound}}
\end{lemma}
The bounds in~\cite{HeyMat20} are formulated only for $p<p_c$, but as the bounds are increasing in $p$, a limit argument easily extends them to the critical point. We now show how we represent the bound in~\eqref{eq:prelim:Pi1_diagrammatic_bound} in terms of pictorial diagrams. As the bound on $\Pi_p^{(2)}$ is even longer to write down, Lemma~\ref{lem:prelim:diagrammatic_bound_Pi2} is stated only in terms of these pictorial bounds.

The points $w,u,t,z,x$ summed over are represented as squares, factors of $\taup$ are represented as lines, and lines with a `$\bullet$' (`$\circ$') symbol represent factors of $\taupf$ ($\taupo$). For example, the factor $\taup(w-u)$ is represented as a line between two squares, which we think of as the points $w$ and $u$. We interpret the factor $\taup(u)$ as a line between $u$ and the origin. We indicate the position of $u$ and $x$ in the below diagrams. After expanding the two cases in~\eqref{eq:prelim:Pi1_diagrammatic_bound} according to whether $|\{t,z,x\}|=3$ or $|\{t,z,x\}|=1$, this pictorial representation allows us to rewrite the bound in~\eqref{eq:prelim:Pi1_diagrammatic_bound} as
	\al{ \sum_{x\in\Zd} \Pi_p^{(1)}(x) &\leq p^2 \sum_{w,u,t,z,x \in\Zd} \taupf(w) \taup(u)\taup(w-u) \taupo(z-w) \taupf(t-u) \taup(z-t) \taup(x-t) \taup(x-z) \\
		& \qquad + p \sum_{w,u,x \in\Zd} \taupf(w) \taup(u)\taup(w-u) \taupo(x-w) \taup(x-u) \\
		& \leq p^2 \sum  \mathrel{\raisebox{-0.25 cm}{\includegraphics{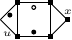}}}
				\ + p \sum \mathrel{\raisebox{-0.25 cm}{\includegraphics{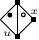}}}.  }
We now formulate the bound on $\Pi_p^{(2)}$; more precisely, we are going to insert a case distinguishing indicator, resulting in two bounds.

\begin{lemma}[{{Diagrammatic bound on $\Pi_p^{(2)}$, \cite[Lemma 3.10]{HeyMat20}}}] \label{lem:prelim:diagrammatic_bound_Pi2}
Let $p\in[0,p_c]$. Then
	\algn{ \sum_{u,v,x\in\Zd} \pp & \Big( \{\orig \Longleftrightarrow u\}_0 \cap E'(u,v;\C_0)_1 \cap E'(v,x;\C_1)_2 \cap\big(\{v\notin\thinn{\C_0}\}_0 \cup \{x\notin\thinn{\C_1}\}_1 \big) \Big) \notag\\
		& \leq p^5 \sum \mathrel{\raisebox{-0.25 cm}{\includegraphics{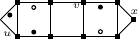}}} 
				\ + p^4 \sum \mathrel{\raisebox{-0.25 cm}{\includegraphics{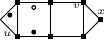}}} 
				\ + p^4 \sum \mathrel{\raisebox{-0.25 cm}{\includegraphics{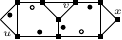}}} \notag\\
		& \qquad \ + p^3 \sum \mathrel{\raisebox{-0.25 cm}{\includegraphics{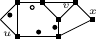}}}
				\ + p^3 \sum \mathrel{\raisebox{-0.25 cm}{\includegraphics{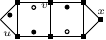}}}  \label{eq:prelim:Pi2_diagrammatic_bound_A}}
and 
	\algn{ \sum_{u,v,x\in\Zd} \pp & \Big( \{\orig \Longleftrightarrow u\}_0 \cap E'(u,v;\C_0)_1 \cap E'(v,x;\C_1)_2 \cap \{v\in\thinn{\C_0}\}_0 \cap \{x\in\thinn{\C_1}\}_1 \Big) \notag\\
		& \leq  p^2 \sum \mathrel{\raisebox{-0.25 cm}{\includegraphics{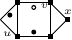}}}. \label{eq:prelim:Pi2_diagrammatic_bound_B}}	
\end{lemma}

\subsection{Convolution bounds}
The last result from~\cite{HeyMat20} we need to state is going to be important for the proofs of Section~\ref{sec:exp:bounds_l_connection}.

\begin{lemma}[{{Bounds on convolutions of $\jeq$ and $\taup$, \cite[Lemma 4.6]{HeyMat20}}}] \label{lem:prelim:J_tau_convolutions}
Let $m,n\in\N_0$ with $2m+n\geq 2$. For $p\in[0,p_c]$ and $d>20n/9$,
	\[ \sup_{a\in\Zd} p^{2m+n-1} \big(\jeq^{\ast 2m} \ast \taup^{\ast n}\big)(a) \leq c \Omega^{1-m} \]
for some constant $c=c(m,n)$.
\end{lemma}
Lemma 4.6 in \cite{HeyMat20} states only the upper bound $c\Omega^{-1}$, but an inspection of its proof gives the stronger bound of Lemma \ref{lem:prelim:J_tau_convolutions}: for $m\ge4$ this is evident from the first bound on page 842 in  \cite{HeyMat20}, and for $m\le4$ one has to adapt  \cite[(4.10)]{HeyMat20} and the subsequent lines accordingly. 
 
Again, Lemma~4.6 in~\cite{HeyMat20} is stated only for $p<p_c$, but the bounds
	\[ 2dp_c \leq 1 + \omone \quad \text{and} \quad  \sup_{k\in\fspace} \frac{p_c|\widehat\tau_{p_c}(k)|}{\widehat G_1(k)} \leq 1 + \omone \]
are sufficient for the statement to extend to $p_c$. 
While the former bound is a direct consequence of Proposition~\ref{thm:prelim:convergence_of_LE}, the latter bound (for $k\neq 0$) follows from the infra-red bound~\eqref{eq:prelim:nfrared} and $|\fconnf(k)| \leq 1$. 
The bound for $k=0$ follows from the continuity of the Fourier transform.

\section{Proof of Theorem~\ref{thm:expansion_of_p_c}} \label{sec:expansion_critical_point}
In this section, we prove Theorem~\ref{thm:expansion_of_p_c} assuming Lemma~\ref{lem:exp:Pis_expansion}, the latter providing an asymptotic expansion of the lace-expansion coefficients $\Pi^{(0)}, \Pi^{(1)}$, and $\Pi^{(2)}$ up to order $\omtwo$.

\begin{lemma}[Expansion of lace-expansion coefficients] \label{lem:exp:Pis_expansion}
As $d\to\infty$,
	\al{ \widehat\Pi^{(0)}_{p_c}(0) &= \tfrac 12 \Omega^2p_c^2 + \tfrac 52 \Omega^{-1} + \omtwo, \\ \widehat\Pi^{(1)}_{p_c}(0) &= \Omega p_c + 2 \Omega^2 p_c^2 + 4\Omega^{-1} + \omtwo, \\
			 \widehat\Pi^{(2)}_{p_c}(0) &= 10 \Omega^{-1} + \omtwo.}
\end{lemma}

Lemma~\ref{lem:exp:Pis_expansion} is the union of Lemmas~\ref{lem:exp:Pi0_finer},~\ref{lem:exp:Pi1_finer},~\ref{lem:exp:Pi2_finer}, which are proved in Section~\ref{sec:exp:bounds_on_Pis}. As a preparation for these proofs, we need Section~\ref{sec:exp:bounds_l_connection}. These proofs are lengthy considerations of numerous percolation configurations in search for contributions of the right order of magnitude (in terms of powers of $\Omega^{-1}$). They are very mechanical in that they boil down to counting exercises and case distinctions. This also means that no new ideas are needed to extend Lemma~\ref{lem:exp:Pis_expansion} to higher orders of $\Omega^{-1}$ and expand the higher-order coefficients $\widehat\Pi^{(3)}, \widehat\Pi^{(4)}$, etc. The necessary effort increases exponentially however.

\begin{proof}[Proof of Theorem~\ref{thm:expansion_of_p_c}]
Let first $p<p_c$. Taking the Fourier transform of~\eqref{eq:prelim:OZE} and solving for $\ftau$ at $k=0$ gives
	\eqq{ p\ftau(0) = \frac{p\Omega + p\fpip(0) }{1- p(\Omega + \fpip(0))}. \label{eq:exp:OZE_fourier}}
A standard result is that $p\ftau(0)=\E_p[|\C(\orig)|]-1$ diverges as $p \nearrow p_c$, cf.\ \cite{AizenNewma84}. 
As the numerator of~\eqref{eq:exp:OZE_fourier} is bounded by $1+\omone$, we conclude that $p_c$ satisfies
	\eqq{ 1 - p_c(\Omega + \widehat\Pi_{p_c}(0)) = 0. \label{eq:exp:pc_denominator_diverging_identity}}
From here on out, we abbreviate $\widehat\Pi = \widehat{\Pi}_{p_c}(0)$ and $\widehat\Pi^{(m)} = \widehat\Pi_{p_c}^{(m)}(0)$. We know from Proposition~\ref{thm:prelim:convergence_of_LE} that $|\widehat\Pi/\Omega| = \omone$, and so rearranging~\eqref{eq:exp:pc_denominator_diverging_identity} yields
	\eqq{\Omega p_c = \frac{1}{1 + \widehat\Pi/\Omega} = 1+ \omone. \label{eq:exp:Omega_pc_first_identity}}
Proposition~\ref{thm:prelim:convergence_of_LE} moreover provides the bound $|\widehat\Pi^{(m)}| = \mathcal O(\Omega^{1-(m\vee 1)})$ for all $m\geq 0$. We can use this to describe $\Omega p_c$ in more detail as
	\begin{align} \Omega p_c &= 1 - \frac{\widehat\Pi^{(0)}/\Omega - \widehat\Pi^{(1)}/\Omega + \widehat\Pi^{(2)}/\Omega + \sum_{m\geq 3} (-1)^m\widehat\Pi^{(m)} /\Omega}{1+\widehat\Pi/\Omega} \notag\\
				&= 1 - \frac{\widehat\Pi^{(0)}/\Omega - \widehat\Pi^{(1)}/\Omega + \widehat\Pi^{(2)}/\Omega}{1+\widehat\Pi/\Omega} + \mathcal O(\Omega^{-3}). \label{eq:exp:Omega_pc_expansion}
	\end{align}
Simplifying~\eqref{eq:exp:Omega_pc_expansion} to an error term of order $\omtwo$ gives
	\eqq{\Omega p_c = 1 - \widehat\Pi^{(0)}/\Omega + \widehat\Pi^{(1)}/\Omega + \omtwo \label{eq:exp:critical_point_first_level}.}
Plugging in the expansion for $\widehat\Pi^{(0)}$ and $\widehat\Pi^{(1)}$ from Lemma~\ref{lem:exp:Pis_expansion} gives $\Omega p_c = 1 + \tfrac 52 \Omega^{-1} + \omtwo$. Using this and the first identity of~\eqref{eq:exp:Omega_pc_first_identity} in~\eqref{eq:exp:Omega_pc_expansion} gives
	\eqq{ \Omega p_c = 1 - \big(\widehat\Pi^{(0)}/\Omega - \widehat\Pi^{(1)}/\Omega + \widehat\Pi^{(2)}/\Omega \big) \big(1 + \tfrac 52 \Omega^{-1} + \omtwo \big)+ \mathcal O(\Omega^{-3}).
				 \label{eq:exp:Omega_pc_expansion_refined}}
Applying Lemma~\ref{lem:exp:Pis_expansion} to~\eqref{eq:exp:Omega_pc_expansion_refined} proves the theorem.
\end{proof}


\section{Further bounds on connection events} \label{sec:exp:bounds_l_connection}
This section extracts some results that are frequently used in the proofs of Section~\ref{sec:exp:bounds_on_Pis}. We start by defining $l$-step connections.
\begin{definition}[$l$-step connections] Let $l\in\N$ and $p\leq p_c$. \label{def:exp:l_step_connections}
\begin{enumerate}
\item We define $\{u \lconn{l} v\}$ as the event that $u$ is connected to $v$ via an occupied and self-avoiding path of length at least $l$ (shorter occupied paths might be present as well), and let $\taup^{(l)} = \pp(u \lconn{l} v)$.

We define $\{u \lconn{ \geq l} v\}$ as the event that $u$ is connected to $v$ but there is no occupied path from $u$ to $v$ of length less than $l$. 
Furthermore, let $\{u \lconn{\leq l} v\}$ be the event that $u$ and $v$ are connected by an occupied path of length at most $l$. Lastly, set $\{u \lconn{= l} v\} := \{u \lconn{\leq l} v\} \cap \{u \lconn{\geq l} v\}$.

\item We define $\{u \lcyc{l} v\}:=\cup_{j=1}^{l-1} \{u \lconn{j} v\} \circ\{u \lconn{l-j} v\}$ as the event that $u$ and $v$ lie in a cycle of length at least $l$, where all sites---except possibly $u$ and $v$---are occupied.

Let $\{ u \lcyc{\geq l} v\}$ be the event that $\{u \Longleftrightarrow v\}$ and the shortest cycle containing $u$ and $v$ (with all other vertices occupied) is of length at least $l$. Similarly, let $\{ u \lcyc{\leq l} v\}$ be the event that $\{u \Longleftrightarrow v\}$ and the shortest cycle containing $u$ and $v$ is of length at most $l$, and let $\{ u \lcyc{= l} v\}:=\{ u \lcyc{\geq l} v\} \cap \{ u \lcyc{\leq l} v\}$. 
\item Also, define
	\al{ \triangle^{(l)} (u,v,w) &:= \sum_{\substack{l_1, l_2, l_3 \geq 1: \\ l_1+l_2+l_3=l}} \taup^{(l_1)}(u)\taup^{(l_2)}(v-u) \taup^{(l_3)}(w-v), \\
		\dtrial^{(l)} (u,t,z,x) &:= \sum_{\substack{l_1, l_2 \geq 0, l_3 \geq 3: \\ l_1+l_2+l_3=l-1} } \big(\delta_{t,u}\delta_{\orig,l_1} + p(1-\delta_{\orig,l_1})\taup^{(l_1)}(t-u) \big)
							\big(\delta_{\orig,z}\delta_{\orig,l_2} + (1-\delta_{\orig,l_2})\taup^{(l_2)}(z) \big) \\
		& \hspace{3cm} \times \jeq(u) \triangle^{(l_3)}(t-z,x-z,\orig).}
\end{enumerate}
\end{definition}

See Figure~\ref{fig:dtri9} for an illustration of $\dtri^{(l)}$. We remark that $\taup^{(1)}=\taup$. Moreover, note that $\Zd$ is bipartite and thus contains no cycles of odd length, which is why $\{u \lcyc{2l-1} v\} = \{u \lcyc{2l} v\}$ and $\triangle^{(2l-1)} (u,v,0) = \triangle^{(2l)} (u,v,0)$.
\begin{figure}
         \centering
         \includegraphics[scale=1]{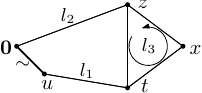}
	\caption{An illustration of the diagrammatic quantity $\dtri^{(l)}$. The `$\sim$' symbol on the line between $\orig$ and $u$ means that $|u|=1$.}
         \label{fig:dtri9}
\end{figure}

The bounds stated in Lemma~\ref{lem:exp:l_step_connections} provide the core tools in dealing with lower-order terms in the bounds on $\Pi^{(i)}$ in the proofs of Section~\ref{sec:exp:bounds_on_Pis}.

\begin{lemma}[Bounds on $l$-step connection probabilities] \ \label{lem:exp:l_step_connections}
Let $2\leq l\in\N, x\in\Zd$ and $p\leq p_c$. Then
	\eqq{ \taup^{(l)}(x) = \mathcal O \big( |x| \Omega^{1- (l+|x|)/2 }  \big). \label{eq:exp:tau_l_bound}}
Moreover,
	\eqq{\sum_{x\in\Zd} \pp ( \orig \lcyc{2l} x) \leq p \sum_{u,x\in\Zd} \triangle^{(2l)}(u,x,\orig) = \mathcal O\big(\Omega^{2-l}\big)\label{eq:exp:cycle_triangle_relation}}
and
	\eqq{ p^2 \sum_{u,t,z,x\in\Zd} \dtrial^{(9)}(u,t,z,x) = \omtwo.\label{eq:exp:double_triangle_bound}}
\end{lemma}
\begin{proof}
We observe that
	\[ \taup^{(l)}(x) \leq \sum_{y \in \Zd} J(y) \pp(y \text{ occupied, } y\lconn{l-1}x) = p (J \ast \taup^{(l-1)})(x).\]
Iterating this yields
	\eqq{\taup^{(l)}(x) \leq p^{l-1} (\jeq^{\ast (l-1)} \ast \taup)(x). \label{eq:exp:l_step_conection_taup_bound}}
To prove the first part in~\eqref{eq:exp:cycle_triangle_relation}, note that by the BK inequality,
	\al{\sum_{x\in\Zd} \pp ( \orig \lcyc{2l} x) &\leq \sum_x \sum_{j=1}^{l} \taup^{(j)}(x) \taup^{(2l-j)}(x) \leq \sum_x \sum_{j=1}^{l} \taup^{(j)}(x) p \big(\jeq\ast \taup^{(2l-j-1)}\big)(x) \\
		& \leq p \sum_x \sum_{j=1}^{l} \taup^{(j)}(x) \sum_u \taup^{(1)}(u) \taup^{(2l-j-1)}(x-u) \leq p \sum_{u,x} \triangle^{(2l)} (u,x,\orig). }
To prove the second part of~\eqref{eq:exp:cycle_triangle_relation}, we combine~\eqref{eq:exp:l_step_conection_taup_bound} with Observation~\ref{obs:prelim:tau_J_extraction}, yielding
	\algn{ p \sum_{u,x\in\Zd }\triangle^{(2l)} (u,x,\orig) &\leq p \sum_{\substack{l_1, l_2, l_3: \\ l_1+l_2+l_3=2l}} \sum_{u,x\in\Zd}	p^{l_1-1} \big( \jeq^{\ast l_1-1}\ast\taup \big)(u) \notag\\
		& \hspace{3cm} \times p^{l_2-1} \big( \jeq^{\ast l_2-1}\ast\taup \big)(x-u) p^{l_3-1} \big( \jeq^{\ast l_3-1}\ast\taup \big)(x) \notag\\
		& = p^{2l-2} \sum_{\substack{l_1, l_2, l_3: \\ l_1+l_2+l_3=2l}} \big( \jeq^{\ast 2l-3}\ast\taup^{\ast 3} \big)(\orig) = p^{2l-2} \binom{2l-1}{2} \big( \jeq^{\ast 2l-3}\ast\taup^{\ast 3} \big)(\orig) \notag\\
		& \leq 2l^2 p^{2l-2} \Big( \jeq^{\ast 2l-3}\ast\big(\jeq + p(\jeq\ast\taup)\big)^{\ast 3} \Big)(\orig) \notag\\
		& \leq 6l^2 \sum_{j=0}^{3} p^{2l-2+j} \big( \jeq^{\ast 2l}\ast\taup^{\ast j} \big)(\orig) \leq \mathcal O(\Omega^{2-l}), \label{eq:prelim:l_triangle_bound}}
where the last inequality is due to Lemma~\ref{lem:prelim:J_tau_convolutions}. 

To prove the bound on $\taup^{(l)}$, we first use the bound~\eqref{eq:exp:l_step_conection_taup_bound} and then apply Observation~\ref{obs:prelim:tau_J_extraction} with $n=1$ and $m=|x|+1$ to obtain
	\al{ \taup^{(l)}(x) &\leq \O(1)\, p^{l+|x|} \big(\jeq^{\ast l+|x|} \ast \taup \big)(x) + \sum_{j=0}^{|x|} \O(1)\, p^{l-1+j} \jeq^{\ast l+j}(x). }
The first term, i.e.~the term including a convolution with $\taup$, is bounded using Lemma~\ref{lem:prelim:J_tau_convolutions}. The second term, i.e.~the convolutions over $\jeq$, are bounded using Observation~\ref{obs:prelim:J_convolutions} and \eqref{eq:exp:Omega_pc_first_identity} to get
	\al{\taup^{(l)}(x) &\leq \O(1)\, \Omega^{1-(l+|x|)/2} + \sum_{j=0}^{|x|} \O(1)\Omega^{1-(|x|+l+j)/2} \leq \big(|x|+2\big) \O(1)\, \Omega^{1-(|x|+l)/2}.  }

To prove~\eqref{eq:exp:double_triangle_bound}, we split $\dtri$. First observe that when $l_1=l_2=0$,
	\[ p^2 \sum_{u,t,z,x}\jeq(u) \delta_{t,u}\delta_{\orig,z} \triangle^{(l_3)}(t-z,x-z,\orig) \leq p^2 \sum_{u,x} \triangle^{(l_3)}(u,x,\orig),\]
which is in $\omtwo$ for $l_3 =9$. Let next $l_1 \neq 0 = l_2$. Then
	\[ p^3 \sum_{u,t,x} \jeq(u) \taup^{(l_1)}(t-u) \triangle^{(l_3)}(t,\orig,x) \leq p^3 \sum \mathrel{\raisebox{-0.25 cm}{\includegraphics{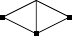}}}
				\ \leq \tripf \trip = \omtwo.\]
When $l_1 = 0 \neq l_2$,
	\eqq{ p^2 \sum_{u,z,x} \jeq(u) \taup^{(l_2)}(z) \triangle^{(l_3)}(u-z,x-z,\orig) = p^2 \sum_{u,z} \triangle^{(l_3)}(\orig,z,u) (\jeq\ast\taup^{(l_2)})(u-z).  \label{eq:exp:dtri_proof_A}}
If $l_3 \geq 5$, then~\eqref{eq:exp:dtri_proof_A} is bounded by $p\tripf \sum_{u,x} \triangle^{(5)}(\orig,u,x) = \omtwo$. If $l_3 \leq 4$, then $l_2 \geq 4$. We can rewrite the left-hand side of~\eqref{eq:exp:dtri_proof_A} as
	\[ p^2\sum_{u,z} \sum_{\substack{m_1,m_2,m_3:\\ m_1+m_2+m_3=l_3}} \jeq(u) \taup^{(l_2)}(z) \taup^{(m_1)}(z-u) (\taup^{(m_2)}\ast\taup^{(m_3)})(z-u) \leq p \tripf \sum_{u,z} \triangle^{(6)}(\orig,u,z)
				=\omtwo, \]
as $l_2+m_1 \geq 5$.

Lastly, let $l_1 \neq 0 \neq l_2$. If $l_3 \geq 5$, then
	\al{p^3 \sum_{u,t,z,x}\jeq(u) & \taup^{(l_1)}(z) \taup^{(l_2)}(t-u) \triangle^{(l_3)}(t-z,x-z,\orig) \\
		& = p^3 \sum_{t,z} \triangle^{(l_3)} (t,z,\orig) (\taup^{(l_1)}\ast\jeq\ast\taup^{(l_2)})(z-t) \leq p \trip \sum_{t,z} \triangle^{(6)}(t,z,\orig) = \omtwo. }
If $l_3 \leq 4$, then $l_1+l_2 \geq 4$. We bound
	\al{p^3 \sum_{u,t,z,x}\jeq(u) & \taup^{(l_1)}(z) \taup^{(l_2)}(t-u) \triangle^{(l_3)}(t-z,x-z,\orig) \leq p^2 \tripf (\jeq\ast\taup^{(l_2)}\ast\taup\ast\taup^{(l_1)})(\orig)  \\
		& \leq \tripf \big( p^4 (\jeq^{\ast 3} \ast \taup^{\ast 3})(\orig) \big) = \omtwo, }
where we used the same sequence of bounds as in~\eqref{eq:prelim:l_triangle_bound}.
\end{proof}

Lastly, we state an observation that appears enough times throughout the arguments of Section~\ref{sec:exp:bounds_on_Pis} for us to extract and state it here.

\begin{observation} \label{obs:exp:pivotality}
Let $a\in\Zd$. Let further $u\neq v$ be two neighbors of $a$, and set $t=v+u-a$. Then
	\[ E'(u,v;\{a\}) \cap \big( \{t=a\} \cup \{t \text{ is vacant}\} \big) \subseteq \{u \lconn{4} v\}.  \] 
\end{observation}
\begin{proof}
Let $A= \{a\}$. We know that $E'(u,v;A) \subset \{u \longleftrightarrow v\}$. If $a$ is vacant, then the shortest possible $u$-$v$-path that may be occupied is of length $4$ and the claim holds.

On the other hand, if $a$ is occupied, then $\{u \longleftrightarrow v\}$ holds. However, $\{u \throughconn{A} a\}$ also holds, and so for $E'(u,v;A)$ to hold, $a$ cannot be a pivotal vertex. But in order for $a$ not to be pivotal, there needs to be a second $u$-$v$-path, avoiding $a$. But either $t$ is vacant, or $t=a$; in both cases, a second $u$-$v$-path must be of length at least $4$, proving the claim.
\end{proof}

\section{Detailed analysis of the first three lace-expansion coefficients} \label{sec:exp:bounds_on_Pis}

\subsection{Analysis of $\widehat\Pi^{(0)}$}
We recall that we write $\widehat\Pi^{(i)} = \widehat\Pi_{p_c}^{(i)}(\orig)$. We will also abbreviate $\p=\p_{p_c}$ and $\tau=\tau_{p_c}$ throughout Section~\ref{sec:exp:bounds_on_Pis}. We use~\eqref{eq:exp:Omega_pc_first_identity} a lot throughout Section~\ref{sec:exp:bounds_on_Pis}, and we recall that it states
	\[ \Omega p_c= 1+ \omone\]
and follows from Proposition~\ref{thm:prelim:convergence_of_LE}. Moreover, we will use~\eqref{eq:exp:tau_l_bound} of Lemma~\ref{lem:exp:l_step_connections} frequently in the proofs to follow and will not mention every time we do so.

\begin{lemma}[Finer asymptotics of $\widehat\Pi^{(0)}$] \label{lem:exp:Pi0_finer}
As $d\to\infty$,
	\[ \widehat\Pi^{(0)} =  \tfrac 12 \Omega^2 p_c^2 + \tfrac 52 \Omega^{-1} + \omtwo.\]
\end{lemma}
\begin{proof}
Recall that $\widehat{\Pi}^{(0)} = \sum_{x} \p(\orig \Longleftrightarrow x) - \jeq(x)$. This sum only gets contributions from $|x| \geq 2$. Now,
	\al{\widehat\Pi^{(0)}= \sum_{|x|\geq 2}\p(\orig \Longleftrightarrow x) &= \sum_{|x|=2}\p \big(\orig \lcyc{\leq 4} x\big) + \sum_{|x| \leq 3}\p\big(\orig \lcyc{= 6} x\big) + \sum_{|x|\geq 2} \p\big(\orig \lcyc{\geq 8} x\big) \\
		& = \sum_{|x|=2}\p \big(\orig \lcyc{\leq 4} x\big) + \sum_{|x| \leq 3}\p\big(\orig \lcyc{= 6} x\big) + \omtwo,  }
where the last identity is due to Lemma~\ref{lem:exp:l_step_connections}. We first consider $4$-cycles. The only points $x$ with $|x| \geq 2$ that can form a $4$-cycle with the origin are those with $|x|=2, \|x\|_\infty=1$. There are $\tfrac 12 \Omega(\Omega-2)$ such points. If $x=v_1+v_2$ (with $|v_i|=1$) is such a point, then $\{\orig \lcyc{\leq 4} x\}$ holds if and only if $\{v_1, v_2\} \subseteq \omega$. Therefore,
	\eqq{ \sum_{|x| \geq 2} \p \big(\orig \lcyc{\leq 4} x\big) = \tfrac 12 \Omega(\Omega-2) p_c^2 = \tfrac 12 \Omega^2 p_c^2 - \Omega^{-1} + \omtwo. \label{eq:exp:Pi0_4cycles}}
We are left to consider points $|x| \geq 2$ contained in cycles of length $6$ that also contain the origin. Note that this is possible for $|x|\in\{2,3\}$ and $\| x\|_\infty \in\{1,2\}$. We first claim that $\|x\|_\infty=2$ gives a contribution of order $\omtwo$.

Indeed, there are $\Omega$ points $x$ with $|x| = 2$ and $\|x\|_\infty=2$, and any such point is contained in at most $c \Omega$ many origin-including cycles of length $6$ (where $c$ is some absolute constant). Any given $6$-cycle has probability $p_c^4$ of being present, and so the contribution is at most $c \Omega^2 p_c^4 = \omtwo$.

Similarly, there are at most $\Omega(\Omega-2)$ points $x$ with $|x|=3, \|x\|_\infty=2$, and any such point is contained in exactly one origin-including cycle of length $6$. Hence, this contributes at most $\Omega^2 p_c^4 = \omtwo$ as well.

Let now $|x|=3, \|x\|_\infty=1$. There are $\tfrac 16\Omega(\Omega-2)(\Omega-4)$ such points. Such a point spans a ($3$-dimensional) cube with the origin, in which two internally disjoint paths of respective length $3$, making up the sought-after $6$-cycle, have to be occupied. There are $9$ such cycles. By inclusion-exclusion,
	\eqq{ \sum_{|x|=3, \|x\|_\infty=1} \p (\orig \lcyc{= 6} x) \begin{cases} \leq \tfrac 96 \Omega^3 p_c^4 = \tfrac 32 \Omega^{-1} + \omtwo, \\
								  \geq \tfrac 16 (\Omega-4)^3 \big[9p_c^4 - \binom{9}{2} p_c^5 \big] = \tfrac 32 \Omega^{-1} + \omtwo   \end{cases}  \label{eq:exp:Pi0_6cycles_cubes}}
(for the lower bound, we sum the probabilities for the 9 cycles to be occupied and substract the probability that at least two of them are occupied at the same time). 
Lastly, consider one of the $\tfrac 12\Omega(\Omega-2)$ points $x=v_1+v_2$ with $|x|=2, \|x\|_\infty=1$, and $|v_i|=1$. Note that there are precisely two paths of length $2$ from $\orig$ to $x$, namely the ones using $v_i$. To produce a relevant contribution to $\{\orig \lcyc{= 6} x\}$, we claim that exactly one of the two vertices must be vacant and the other occupied. Indeed, if both are occupied, then there is a $4$-cycle containing $\orig$ and $x$. If both are vacant, then the shortest possible cycle containing $\orig$ and $x$ is of length $8$.

We assume $v_1$ to be occupied and $v_2$ to be vacant (the reverse gives the same contribution by symmetry, and we respect it with a factor of $2$). It remains to count the number of paths of length $4$ from $\orig$ to $x$ that avoid $v_1$ and $v_2$. Avoiding $\pm v_i$ gives $\Omega-4$ options for the first step. There are two options for the second step (namely, to a neighbor of $v_1$ or $v_2$). Steps $3$ and $4$ are now fixed: Out of the two shortest paths to $x$, one is via $v_i$, and is not an option. In conclusion, the probability that there is a $\orig$-$x$-path of length $4$ traversing some fixed neighbor of $\orig$ (which is not $\pm v_i$) first is $p_c^2 (2p_c-p_c^2)$. This gives
	\eqq{ \sum_{|x|=2, \|x\|_\infty=1} \p (\orig \lcyc{= 6} x) \begin{cases} \leq  \tfrac 12 \Omega^3 4 p_c^4 = 2 \Omega^{-1} + \omtwo, \\
				\geq (\Omega-4)^3 p_c^3 (2p_c - p_c^2) - 4\Omega^2 p_c \binom{\Omega-4}{2} p_c^6 = 2 \Omega^{-1} + \omtwo,\end{cases} \label{eq:exp:Pi0_6cycles_remaining}.}
Summing up~\eqref{eq:exp:Pi0_4cycles},~\eqref{eq:exp:Pi0_6cycles_cubes}, and~\eqref{eq:exp:Pi0_6cycles_remaining} finishes the proof.
\end{proof}

\subsection{Analysis of $\widehat\Pi^{(1)}$}
\begin{lemma}[Finer asymptotics of $\widehat\Pi^{(1)}$] \label{lem:exp:Pi1_finer}
As $d\to\infty$,
	\[ \widehat\Pi^{(1)} =  \Omega p_c + 2\Omega^2 p_c^2 + 4 \Omega^{-1} + \omtwo.\]
\end{lemma}
\begin{proof}
 Abbreviating $\C_0 = \widetilde\C^{u}(\orig;\omega_0)$, we recall that
	\eqq {\widehat\Pi^{(1)} = p_c \sum_{u\in\Zd}\sum_{x\in\Zd} \p \big( \{\orig \Longleftrightarrow u\}_0 \cap E'(u,x; \C_0)_1\big). \label{eq:exp:Pi1_definition}}
While this is a double sum over all points in $\Zd$, we first prove that only small values of $u$ give relevant contributions. To this end, assume that $|u| \geq 3$. We use the pictorial representation of the bound in Lemma~\ref{lem:prelim:diagrammatic_bound_Pi1} and decompose it in terms of modified triangles introduced in Definition~\ref{def:prelim:modified_two-point_functions}. In the below pictorial diagrams, points over which the supremum is taken (in particular, those points are \emph{not} summed over) are represented by colored disks. The indicator that two such points (disks) may not coincide is represented by a disrupted two-sided arrow. Lemma~\ref{lem:prelim:diagrammatic_bound_Pi1} together with Proposition~\ref{thm:prelim:triangle_bounds} then gives
	\algn{ 
	&\hspace{-5em}p_c \sum_{|u|\ge3}\sum_{x\in\Zd} \p \big( \{\orig \Longleftrightarrow u\}_0 \cap E'(u,x; \C_0)_1\big) \notag \\
		& \leq p_c \sum \mathds 1_{\{|u|\geq 3\}} \Big( p_c \mathrel{\raisebox{-0.25 cm}{\includegraphics{N1_exp_A_general_labels.pdf}}}
				\ + \mathrel{\raisebox{-0.25 cm}{\includegraphics{N1_exp_B_general_labels.pdf}}} \Big) \notag\\
		& \leq \sum \Big( \mathds 1_{\{|u|\geq 3\}} \mathrel{\raisebox{-0.25 cm}{\includegraphics{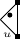}}}
			\Big( \sup_{\textcolor{darkorange}{\bullet}, \textcolor{blue}{\bullet}} p_c \sum \mathrel{\raisebox{-0.25 cm}{\includegraphics{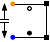}}}
			\Big( \sup_{\textcolor{green}{\bullet}, \textcolor{altviolet}{\bullet}} p_c \sum \mathrel{\raisebox{-0.25 cm}{\includegraphics{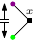}}} \Big)\Big)\Big) \notag \\
		& \qquad + \sum \Big( \mathds 1_{\{|u|\geq 3\}} \mathrel{\raisebox{-0.25 cm}{\includegraphics{N1_exp_A_bound1.pdf}}}
			\Big( \sup_{\textcolor{green}{\bullet}, \textcolor{altviolet}{\bullet}} p_c \sum \mathrel{\raisebox{-0.25 cm}{\includegraphics{N1_exp_A_bound3.pdf}}} \Big)\Big)
			+   p_c \sum \mathds 1_{\{|u|\geq 3\}} \mathrel{\raisebox{-0.25 cm}{\includegraphics{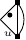}}} \notag\\
		& \leq \big(\triangle_{p_c}^{\bullet\circ} \triangle_{p_c}^{\bullet} + \triangle_{p_c}^{\bullet} + p_c \big)
				\sum \mathds 1_{\{|u|\geq 3\}} \mathrel{\raisebox{-0.25 cm}{\includegraphics{N1_exp_A_bound1.pdf}}} \notag \\
		& \leq \omone \Big(\sum_u \p (\orig \lcyc{6} u) + p_c \sum_{u,w} \triangle^{(6)}(u,w,\orig)  \Big) = \omtwo, \label{eq:exp:Pi1_diag_bound} }
where the last identity is due to Lemma~\ref{lem:exp:l_step_connections}. When we encounter similar diagrams to the ones in~\eqref{eq:exp:Pi1_diag_bound} at later stages of this paper, we decompose them in the same way as performed in~\eqref{eq:exp:Pi1_diag_bound}, but in less detail.

We consider the cases of $|u| \in \{1,2\}$ separately. For both, we make further case distinctions according to the value of $|x|$. The contributions are summarized in the following table:
\begin{center}
\begin{tabular}{ c||c|c|c|c } 
 $\widehat{\Pi}^{(1)}$:  & $x=\orig$ & $|x|=1$ & $|x|=2$ & $|x|=3$ \\ 
 \hline \hline
 \rule{0pt}{3ex} $|u|=1$ & $\Omega p_c$ & $\Omega^2 p_c^2 - 2 \Omega^{-1}$ & $\Omega^2 p_c^2 + \Omega^{-1}$ & $2\Omega^{-1}$ \\
 $|u|=2$ &  & $\Omega^{-1}$ & $\Omega^{-1}$ & $\Omega^{-1}$ \\ 
\end{tabular}
\end{center}

\paragraph{Contributions of $|u|=1$.} By rotational symmetry, we can drop the sum over $u$, and rewrite~\eqref{eq:exp:Pi1_definition} as
	\begin{align} & \ p_c \sum_{|u|=1}\sum_{x\in\Zd} \p \big( \{\orig \Longleftrightarrow u\}_0 \cap E'(u,x; \C_0)_1 \big) \notag \\
			= & \ p_c \sum_{u,x\in\Zd} \jeq(u) \p \left( E'(u,x; \C_0)_1 \right) \label{eq:exp:Pi1_u1_A}\\
			= & \Omega p_c  \sum_{x\in\Zd} \p \left( E'(u,x; \C_0)_1 \right)  \label{eq:exp:Pi1_u1_B}. \end{align}
In~\eqref{eq:exp:Pi1_u1_B} and in the following, we take $u$ to be an arbitrary (but fixed) neighbor of the origin. We recall that $\omega_i$ is a sequence of independent percolation configurations and an event with subscript $i$ takes place on $\omega_i$. Moreover, $E'(u,x; \C_0)$ is indexed to take place on configuration $1$, which is only accurate if $\C_0$ is regarded as a fixed set; otherwise the event takes place on $\omega_0$ and $\omega_1$.

We proceed by splitting the sum over $x$ in~\eqref{eq:exp:Pi1_u1_B} (respectively,~\eqref{eq:exp:Pi1_u1_A}) into different cases. 
	
\underline{The case of $|u|=1, x=\orig$ contributes $\Omega p_c$:} The event $E'(u,x;\C_0)_1$ in~\eqref{eq:exp:Pi1_u1_B} holds, the sum collapses to $1$, and the contribution is $\Omega p_c$.

\underline{The case of $|u|=1=|x|$ contributes $\Omega^2p_c^2 - 2\Omega^{-1} + \omtwo$:} There are $\Omega-1$ choices for $x\neq u$. We exclude the special case $x=-u$ first. For other choices of $x$, we let $v:=x+u$.
\begin{compactitem}
\item For $x=-u$, we have $E'(u,x;\C_0)_1 \subseteq \{u \lconn{4} x\}_1$ by Observation~\ref{obs:exp:pivotality}. Hence,~\eqref{eq:exp:Pi1_u1_B} is bounded by
	\[ \Omega p_c \tau^{(4)}(u-x) = \omtwo.\]
\item Let $x \neq \pm u$ and $v\in \omega_1$, so that $\orig,u,v,x$ span a ``square'' in one of the hyperplanes. Note first that there are $\Omega-2$ choices for $x$, and we can treat them equally by symmetry. Since $x\sim \orig$ and hence $x\in\thinn{\C_0}$, we have that on the event $\{v\in\omega_1\}$, the occurrence of $E'(u,x;\C_0)_1$ implies that either $v$ is  not pivotal for $\{u\longleftrightarrow x\}$, or it is pivotal but $v\notin\thinn{\C_0}$:
	\begin{equation*}\label{eq:EprimevExp}
	E'(u,x;\C_0)_1 \cap\{v \in\omega_1\}_1 = \{v \in\omega_1\}_1 \cap \Big( \{v \notin \thinn{\C_0}\}_0 \cup \{v \notin\piv{u,x}\}_1 \Big).
	\end{equation*}
Note that all three appearing events on the right are independent of each other. 
Recalling that $\C_0$ is shorthand for $\tilde\C^u(0,\omega_0)$, we see that $\{0\leftrightarrow v\}_0$ if either $x\in\omega_0$ or if there is an occupied path of length $\ge 4$ in $\omega_0\setminus\{u\}$: 
	\[\p (v \notin \thinn{\C_0}) =1- \p\big( x \in\omega_0\big) - \p\big(\orig \lconn{\geq 4} v \text{ in } \omega_0\setminus\{u\}\big) = 1-p_c+\omtwo. \]
In order for $v$ to be \emph{not} pivotal, there must be a ``second connection'' from $u$ to $x$, either a short one via $\orig$, or via a longer path; that is, 
	\[	 \p(v \notin\piv{u,x}) = \p(\orig \in \omega_1) + \p\big(u \lconn{\geq 4} x \text{ in } \omega_1\setminus\{v\}\big) = p_c + \omtwo. \]
We can now replace the sum over $x$ in \eqref{eq:exp:Pi1_u1_B} by a factor of $(\Omega-2)$ and thus obtain the contribution 
	\al{ \Omega p_c (\Omega-2) \p \big( E'(u,x;\C_0)_1 \cap \{v \in\omega_1\}_1 \big) &= \Omega(\Omega-2) p_c^2 \Big( 1-p_c+p_c - (1-p_c)p_c \Big) + \omtwo \\
			&= (\Omega p_c)^2 (1-p_c) - 2\Omega p_c^2 + \omtwo \\
			&= \Omega^2 p_c^2 - 3 \Omega^{-1} + \omtwo. }
\item Let $x \neq \pm u, v \notin \omega_1$, and $\orig \notin \omega_1$. For $E'(u,x;\C_0)_1$ to hold, there needs to be a $\omega_1$-path between $u$ and $x$. Its pivotal points cannot lie in $\thinn{\C_0}$ however. First, note that any relevant path between $u$ and $x$ is of length $4$, as
	\[ \Omega p_c (\Omega-2) \p\big(E'(u,x;\C_0)_1 \cap \{ u \lconn{\geq 6} x\}_1\big) \leq \Omega^2 p_c \tau^{(6)}(x-u) = \omtwo. \]
We now investigate the $4$-paths from $u$ to $x$ that avoid $\orig$ and $v$---from Lemma~\ref{lem:exp:Pi0_finer}, we already know that there are $2(\Omega-4)$ of them. Let $z$ be one of the $\Omega-4$ unit vectors satisfying $\dim \gen{u,x,z} =3$, where we let $\gen{\cdot}$ denote the span. We denote by $\gamma_1$ and $\gamma_2$ the two $u$-$x$-paths of length 4 that visit $y_1:=u+z$. W.l.o.g., $\gamma_1$ visits $y_2:=y_1+x$ second and $y_3:=y_2-u$ third, whereas $\gamma_2$ visits $z$ second and $y_3$ third. Let $\{\gamma_i \subseteq \omega_1\}$ denote the event that the $3$ internal vertices of $\gamma_i$ are $\omega_1$-occupied. See Figure~\ref{fig:exp:u1_a} for an illustration.

\begin{figure}
     \centering
     \begin{subfigure}[b]{0.3\textwidth}
         \centering
         \includegraphics[scale=1.5]{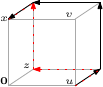}
	\caption{The case $|x|=1$.}
         \label{fig:exp:u1_a}
     \end{subfigure}
     \hfill
     \begin{subfigure}[b]{0.3\textwidth}
         \centering
         \includegraphics[scale=1.5]{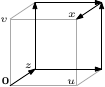}
	\caption{The case $|x|=2, u\sim x$.}
         \label{fig:exp:u1_b}
	\end{subfigure}     
     \hfill
     \begin{subfigure}[b]{0.3\textwidth}
         \centering
         \includegraphics[scale=1.5]{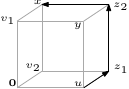}
	\caption{The case $|x|=2, |u-x|=3$.}
         \label{fig:exp:u1_c} 
     \end{subfigure}
     \vspace{0.1cm} \newline 
     \begin{subfigure}[b]{0.3\textwidth}
         \centering
         \includegraphics[scale=1.5]{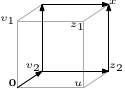}
	\caption{The case $|x|=3$.}
         \label{fig:exp:u1_d}
      \end{subfigure} 
        \caption{An illustration of several appearing cases for $|u|=1$. In the first two cases, $\orig$ and $v$ are vacant in $\omega_1$. In case (a), the black path is $\gamma_1$, the red and dotted one is $\gamma_2$. In case (b), the two $\orig$-$x$-paths are marked as black chains of arrows. In case (c), $\{v_1, v_2\} \cap \omega_0 = \{v_1\}$ and the only relevant $u$-$x$-path is marked in black.}
        \label{fig:exp:u1}
\end{figure}

We now show that only $\gamma_1$ produces a relevant term. Assume first that $y_2 \notin\omega_1$, but $\gamma_2\subseteq\omega_1$. For $E'(u,x;\C_0)_1$ to hold, $z\in\thinn{\C_0}$ must not be a pivotal point. Under $\gamma_2 \subseteq \omega_1$,
	\eqq{ \{ z \notin \piv{u,x} \}_1 \subseteq 
	\bigcup_{a\in\{u,y_1\}, b\in\{y_3, x\} } \{a \longleftrightarrow b \text{ in } \omega_1 \setminus \{z\}\}_1. 
	\label{eq:exp:Pi1_u1_x_1_z_piv}}
Resolving the right-hand side of~\eqref{eq:exp:Pi1_u1_x_1_z_piv} by a union bound gives four connection events. The shortest $\omega_1$-path from $u$ to $x$ of non-vacant vertices is of length 4. Moreover, the shortest $\omega_1$-path from $y_1$ to $y_3$ of non-vacant vertices that avoids $z$ is of length $4$ as well, and so~\eqref{eq:exp:Pi1_u1_B} is bounded by
	\al{\Omega (\Omega-2) p_c \sum_z & \p\Big(E'(u,x;\C_0)_1, \{\orig,v,y_2\}\cap\omega_1 = \varnothing, \gamma_2\subseteq\omega_1 \Big)  \\
		& \leq \Omega^3  p_c^4 \Big( \tau^{(4)}(x-u) + \tau^{(3)}(y_3-u) + \tau^{(3)}(x-y_1) + \tau^{(4)}(y_3-y_1)\Big) = \omtwo.	}
We now show that $\gamma_1\in\omega_1$ gives a contribution. Note that under $\{\orig,v\notin\omega_1, \gamma_1 \subseteq \omega_1\}$, 
	\eqq{ E'(u,x;\C_0)_1 = \bigcap_{i\in\{1,2,3\}} \Big( \{y_i \notin\piv{u,x}\}_1 \cup \{y_i \notin\thinn{\C_0}\}_0\Big) \label{eq:exp:Pi1_u1_x1_gamma1_path}}
But $\p(\{y_i \notin\piv{u,x}\}_1 \cup \{y_i \notin\thinn{\C_0}\}_0) \geq 1-\p(y_i \in \thinn{\C_0}) \geq 1-\tau^{(2)}(y_i) =1-\mathcal O(\Omega^{-1})$ for all $i$ by Lemma~\ref{lem:exp:l_step_connections}, and so, by inclusion-exclusion,
	\al{\Omega (\Omega-2) p_c \sum_z & \p\Big(E'(u,x;\C_0)_1, \{\orig,v\}\cap\omega_1 = \varnothing, \gamma_1\subseteq\omega_1 \Big)  \\
		& \begin{cases} \leq \Omega(\Omega-2)(\Omega-4) (1-p_c)^2 p_c^4 \big(1-\O(\Omega^{-1}) \big) = \Omega^{-1} + \omtwo, \\ 
			 \geq \Omega^3 p_c^4 (1-\omone) - \Omega^2 \binom{\Omega-4}{2} p_c^7 = \Omega^{-1} + \omtwo. \end{cases}	}
\item Let $x \neq \pm u, v\notin\omega_1$, and $\orig\in \omega_1$. By Observation~\ref{obs:exp:pivotality},
	\[ \Omega p_c (\Omega-2) \p \big( E'(u,x;\C_0)_1 \cap \{v \notin\omega_1,\orig\in\omega_1\} \big) \leq \Omega^2 p_c^2 \tau^{(4)}(x-u) = \omtwo. \]
\end{compactitem}

\underline{The case of $|u|=1, |x|=2$ contributes $\Omega^2 p_c^2 +\Omega^{-1} + \omtwo$:}	There are $\tfrac 12 \Omega^2$ choices for $x$. We first consider the $\Omega-1$ choices neighboring $u$ and, among those, exclude the special case $x=2u$ first. For $x$ a neighbor of $u$, we set $v:=x-u$.
\begin{compactitem}
\item Let $x=2u$. Since $x \sim u$, we have $E'(u,x;\C_0)_1 = \{x \in \thinn{\C_0}\}_0 \subseteq \{ \orig \lconn{4} x \}_0$, and so the contribution to~\eqref{eq:exp:Pi1_u1_B} is bounded by $\Omega p_c \tau^{(4)}(x) = \omtwo$.
\item Let $2u\neq x\sim u$ and $v\in \omega_0$. There are $\Omega-2$ choices for $x$. The event $E'(u,x;\C_0)_1$ holds, and so
	\[ \Omega p_c  \sum_{2u\neq x\sim u} \E_0\left[ \mathds 1_{\{v\in\omega_0\}}\p_1\left( E'(u,x; \C_0) \right)  \right] = \Omega(\Omega-2) p_c^2 = \Omega^2 p_c^2 - 2\Omega^{-1} + \omtwo.\]
\item Let $2u\neq x\sim u$ and $v\notin \omega_0$. We partition
	\[ E'(u,x;\C_0)_1 = \Big( E'(u,x;\C_0)_1 \cap \{\orig \lconn{\leq 4} x \text{ in } \Zd\setminus\{u\} \}_0 \Big) \cup \Big(E'(u,x;\C_0)_1 \cap \{\orig \lconn{\geq 6} x \text{ in } \Zd\setminus\{u\}\}_0\Big) \]
and treat the second event by observing
	\[ \Omega p_c  \sum_{2u\neq x\sim u} \p \Big( \{v\notin\omega_0\}_0 \cap \{\orig \lconn{\geq 6} x \text{ in } \Zd\setminus\{u\}\}_0 \cap E'(u,x; \C_0)_1 \Big) 
					\leq \Omega^2 p_c \tau^{(6)}(x) = \omtwo.\]
As the only $2$-paths from $\orig$ to $x$ go through $u$ and $v$ respectively, we can focus on paths of length $4$ avoiding $v$ and $u$. Hence, the status of $v$ is independent of such paths. Let $z$ be one of the $\Omega-4$ neighbors of $\orig$ with $\dim \gen{u,v,z} =3$. For any such $z$, there are two $\orig$-$x$-paths of length $4$ that first visit $z$ and avoid $\{v,u\}$. More precisely, these paths are $(\orig,z,u+z,x+z,x)$ and $(\orig,z,v+z,x+z,x)$. Let $Q_4(z)$ denote the event that at least one of these paths is in $\omega_0$. See Figure~\ref{fig:exp:u1_b} for an illustration. As the events $\{ Q_4(z) \}$ are pairwise independent,
	\al{\{v\notin\omega_0\}_0 \cap  \{\orig \lconn{\leq 4} x \text{ in } \Zd\setminus\{u\}\}_0 \cap E'(u,x; \C_0)_1 &= \{v\notin\omega_0\}_0 \cap \big( \cup_{z} Q_4(z)\big) , \\
		\p(\cup_z Q_4(z)) = (\Omega-4)\p( Q_4(z)) +\O(\Omega^{-4}) &= 2(\Omega-4)p_c^3 + \O(\Omega^{-3}). }
Consequently,
	\al{\Omega p_c  \sum_{2u\neq x\sim u} \p & \Big( \{v\notin\omega_0\}_0 \cap \{\orig \lconn{\leq 4} x \text{ in } \Zd\setminus\{u\}\}_0 \cap E'(u,x; \C_0)_1 \Big) \\
			& = \Omega p_c (\Omega-2)2(\Omega-4) p_c^3 + \omtwo = 2\Omega^{-1} + \omtwo.}
\item Let $|u-x|=3$ and $\|x\|_\infty =2$. There are $\Omega-1$ choices for $x$. Let $2v = x$. Note first that
	\al{ \Omega(\Omega-1)p_c \p \Big( &\big(\{x \in \thinn{\C_0}\}_0 \cup \{ u \lconn{5} x\}_1 \big) \cap E'(u,x;\C_0)_1 \Big) \\
			& \leq \Omega^2 p_c \Big( \tau^{(2)}(x) \tau^{(3)}(x-u) + \tau^{(5)}(x-u)\Big)  = \omtwo.   }
The complementary event is that $x \notin\thinn{\C_0}$ and the presence of a $u$-$x$-path of length $3$. The former implies $v\notin\omega_0$. There are at most four potential sites that can make up internal vertices on a $u$-$x$-path of length $3$, namely $\orig, v, u+v, u+2v$. To avoid potential pivotality of $\orig$ and $v$ and still guarantee a path of length $3$, we require $\{v+u, v+2u\} \subseteq \omega_1$. But both these vertices are of distance at least $2$ from the origin, and at least one of them must be in $\thinn{\C_0}$. In conclusion, 
	\al{ \Omega(\Omega-1)p_c \p \Big( &\{x \notin \thinn{\C_0}\}_0 \cap \{ u \lconn{\leq 3} x\}_1 \cap E'(u,x;\C_0)_1 \Big) \\
			& \leq 2 \Omega^2 p_c \tau^{(2)}(u+v) \tau^{(3)}(x-u) = \omtwo.   }
\item Let $|u-x|=3, \| x\|_\infty=1$, and $x\in\thinn{\C_0}$. Write $x=v_1+v_2$, where $|v_i|=1$. We first show that contributions arise when precisely one point in $\{v_1, v_2\}$ is $\omega_0$-occupied. Note that when both $v_1$ and $v_2$ are vacant in $\omega_0$, the contribution to~\eqref{eq:exp:Pi1_u1_B} is bounded by $\Omega^3 p_c \tau^{(4)}(x) \tau^{(3)}(x-u) = \omtwo$. On the other hand, if $\{v_1,v_2\} \subseteq \omega_0$, then the contribution is bounded by $\Omega^3 p_c^3 \tau^{(3)}(u-x) =\omtwo$.

Let now $v_1\in\omega_0$ and $v_2\notin\omega_0$ (the other case is identical and is respected by counting the contribution twice). There are $\tfrac 12 \Omega^2 (1+\O(\Omega^{-1}))$ choices for $x$. If $\{u \lconn{5} x\}_1$, then the contribution to~\eqref{eq:exp:Pi1_u1_B} is $\omtwo$. Set $z_1=u+v_2, z_2=u+v_2+v_1$, and set $y=u+v_1$. We claim that the only $u$-$x$-path of length 3 that produces a relevant contribution is $(u,z_1,z_2,x)$. See Figure~\ref{fig:exp:u1_c} for an illustration.

First, assume $z_1 \notin \omega_1$. Note that the only other paths of length $3$ from $u$ to $x$ go through either $\orig$ or $y$. But $\{0,y\} \subseteq \thinn{\C_0}$, and so neither $\orig$ nor $y$ can be a pivotal point. Hence, $E'(u,x;\C_0)_1\cap \{z_1\notin \omega_1\}$ enforces $\{\orig,y\} \subseteq \omega_1$. To get to $x$ and avoid pivotality of any points in $\thinn{\C_0}$, at least two points in $\{v_1, v_2, z_1\}$ must be occupied, and the contribution to~\eqref{eq:exp:Pi1_u1_B} is at most
	\[ 2 \Omega p_c \big(\tfrac 12 \Omega^2(1+\omone)\big) p_c^2 \binom{3}{2} p_c^2 = \omtwo.\]
If $z_1\in \omega_1$ and $z_2 \notin\omega_1$, then the only $u$-$x$-path of length $3$ through $z_1$ visits $v_2\in\thinn{\C_0}$. This gives a contribution of $\omtwo$ by the same bound as above. We may turn to the case $z_i \in\omega_1$ for $i\in\{1,2\}$. Now, under $\{v_1\in\omega_0, \{z_1, z_2\} \subseteq \omega_1\}$, we can express $E'(u,x;\C_0)_1$ similarly to~\eqref{eq:exp:Pi1_u1_x1_gamma1_path}, replacing $y_i$ ($i\in[3]$) by $z_i$ ($i \in[2]$). Applying the same bounds, we obtain a contribution to~\eqref{eq:exp:Pi1_u1_B} of
	\[ 2\Omega p_c \big( \tfrac 12 \Omega^2 (1+\omone) \big) \p\big(v_1\in \omega_0, \{z_1, z_2\} \subseteq\omega_1 \big) (1-\omone) = \Omega^{-1} + \omtwo.\]
\item Let $|u-x|=3,\| x\|_\infty=1$, and $x \notin\thinn{\C_0}$. Let $\gamma$ be a $u$-$x$-path in $\omega_1$. By assumption, there needs to be some $z\in\gamma$ with $z \in\thinn{\C_0}$. Consequently, $z$ cannot be a pivotal point and so there needs to be another $u$-$x$-path $\tilde \gamma$ in $\omega_1$ that contains a point $\tilde z \notin \gamma$ with $\tilde z\in\thinn{\C_0}$. 
Assume first that both $\gamma, \tilde \gamma$ are paths of length $3$. If they are disjoint, then the contribution to~\eqref{eq:exp:Pi1_u1_B} is at most $9\Omega^3 p_c^5 = \omtwo$. If they share their first vertex, then, in the terminology of Figure~\ref{fig:exp:u1_c}, it must be either $y$ or $z_1$ (otherwise $\orig$ is pivotal). W.l.o.g., $\tilde\gamma$ must then pass through $z_2$ and so $\tilde z=z_2 \in \thinn{\C_0}$ needs to hold, and the contribution to~\eqref{eq:exp:Pi1_u1_B} is at most $\Omega^3 p_c^4 \tau^{(3)}(z_2) = \omtwo$. Assume next that $\tilde\gamma$ is of length $5$. As $\gamma$ and $\tilde\gamma$ share at most one internal vertex (and there are two internal vertices in $\gamma$), we count a factor of $p_c$ for the unique vertex of $\gamma$, and the contribution to~\eqref{eq:exp:Pi1_u1_B} is at most $18 \Omega^3 p_c^2 \tau^{(5)}(x-u) = \omtwo$. Similarly, when both $\gamma$ and $\tilde\gamma$ are of length at least $5$, the contribution is $\omtwo$.
\end{compactitem}

\underline{The case of $|u|=1, |x|=3$ contributes $2\Omega^{-1} + \omtwo$:} Note that when $\{u \lconn{4} x\}_1$, then the contribution to~\eqref{eq:exp:Pi1_u1_A} is at most
	\eqq{ p_c \sum_{u,x} \triangle^{(8)} (u,x,\orig) + p_c^2 \sum_{u,t,z,x} \dtrial^{(9)} (u,t,z,x) = \omtwo \label{eq:exp:Pi1_u1_x3_ux4}}
by Lemma~\ref{lem:exp:l_step_connections}. We can therefore focus on $x$ with $|x-u|=2$ and $\{u \lconn{ = 2} x\}_1$. Moreover, we can assume that there is no $u$-$x$-path of length $4$. Let $x=u+v_1+v_2$, where $|v_1|=1=|v_2|$, and assume first that $\dim\gen{u,v_1,v_2}=3$. There are $\tfrac 12 (\Omega-2)(\Omega-4)$ choices for $x$. Let $z_i=u+v_i$ be the two internal vertices of the two shortest $u$-$x$-paths---see Figure~\ref{fig:exp:u1_d} for an illustration.

We first claim that only $x\in\thinn{\C_0}$ produces a relevant contribution. Indeed, if $x \notin\thinn{\C_0}$, and as there is no $u$-$x$-path of length $4$, we must have $z_i \in \omega_1 \cap \thinn{\C_0}$ for $i \in\{1,2\}$. For $\{\orig \longleftrightarrow z_i\}_0$ to hold, either $v_i \in\omega_0$, or $\{ \orig \lconn{4} z_i \}_0$, and so~\eqref{eq:exp:Pi1_u1_B} is at most
	\al{ \Omega^3 p_c \p  & \Big(\{\{z_1,z_2\} \subseteq \omega_1\} \cap \big(\{\{v_1,v_2\} \subseteq \omega_0\} \cup \{ \orig \lconn{4} z_1\}_0 \cup \{\orig \lconn{4} z_2\}_0 \big) \Big) \\
		& = \Omega^3 p_c^3 \big( p_c^2 + \tau^{(4)}(z_1) + \tau^{(4)}(z_2) \big) = \omtwo.}
Turning to $x\in \thinn{\C_0}$, note that when $\{z_1, z_2\} \subseteq \omega_1$, then~\eqref{eq:exp:Pi1_u1_B} is at most 
	\[\Omega^3 p_c \p \big( \{\orig \longleftrightarrow x\}_0 \cap \{\{z_1, z_2\} \subseteq \omega_1\}\big) = \Omega^3 p_c^3 \tau^{(3)}(x)=\omtwo. \]
W.l.o.g., we assume that $z_1\in \omega_1$ (and $z_2 \notin \omega_1$) and (by symmetry) count the contribution twice. Now, the contribution to~\eqref{eq:exp:Pi1_u1_B} is equal to
	\eqq{ \Omega(\Omega-2)(\Omega-4) p_c \p \Big(\{x \in \thinn{\C_0}\}_0 \cap \{z_2\notin \omega_1\ni z_1 \}_1 \cap \big( \{z_1 \notin \thinn{\C_0}_0 \cup \{z_1 \notin \piv{u,x}\}_1 \big)   \Big).
				\label{eq:exp:Pi1_u1_x3_ux2}}
If $v_1 \in\omega_0$, then $z_1\in\thinn{\C_0}$ and so $z_1$ cannot be pivotal, which, in turn, forces $\{u \lconn{4} x\}_1$. But this was already shown to produce an $\omtwo$ contribution. Further, if $\{\orig \lconn{5} x\}_0$, then~\eqref{eq:exp:Pi1_u1_x3_ux2} is at most $\Omega^3 p_c^2 \tau^{(5)}(x) = \omtwo$, and so $\orig$ must be $\omega_0$-connected to $x$ by a path of length $3$.

There are precisely two $\orig$-$x$-paths of length 3 that use neither $v_1$ nor $u$, namely $\gamma_1=(\orig,v_2,v_1+v_2,x)$ and $\gamma_2=(\orig,v_2,z_2,x)$. If both are occupied, the contribution is $\omtwo$. Note that
	\[ \p ( z_1 \notin \thinn{\C_0} \mid \gamma_i\subseteq \omega_0 ) \geq 1- 3 \tau^{(2)}(z_1) = 1-\omone, \]
and so~\eqref{eq:exp:Pi1_u1_x3_ux2} becomes
	\[ \Omega^3 (1-\omone) p_c \p \Big( \big( \cup_{i=1,2} \{ \gamma_i\subseteq \omega_0\}_0 \big), z_1 \in\omega_1 \Big) = 2\Omega^3 p_c^4 (1-\omone) = 2\Omega^{-1} + \omtwo.\]
Finally, if $\dim\gen{u,v_1,v_2} \leq 2$, then the same bounds with at least one factor of $\Omega$ in the choice of $x$ gives a contribution of $\omtwo$.

\underline{The case of $|u|=1, |x| \geq 4$ contributes $\omtwo$:}  The bound is the same as in~\eqref{eq:exp:Pi1_u1_x3_ux4}.

\paragraph{Contributions of $|u|=2$.} If $u$ is one of the $\Omega$ points with $|u|=2=\| u \|_\infty$, then $\widehat\Pi^{(1)}$ is bounded by $\Omega p_c\sum_x \p(\orig \Longleftrightarrow u) \taup(u-x)$. For fixed $j=|u-x|$, this is bounded by
	\[ \Omega^{1+j} p_c \tau^{(2)}(u) \tau^{(4)}(u) \tau^{(j)}(x-u) = \omtwo.\]

We now show that we can impose some further restrictions on $u$ and $x$. Recall the bound in \eqref{eq:exp:Pi1_diag_bound}, and observe that if $x\notin\thinn{\C_0}$, then
	\[ p_c \sum_{|u|=2} \sum_{x} \p \big( \{\orig \Longleftrightarrow u\}_0 \cap \{x\notin\thinn{\C_0}\}_0 \cap E'(u,x;\C_0)_1 \big)
			\leq p_c^2 \sum \mathds 1_{\{|u|=2\}} \mathrel{\raisebox{-0.25 cm}{\includegraphics{N1_exp_A_general_labels.pdf}}} \ = \omtwo. \]
Similar considerations enforce that $|x| \leq 3$ and $|x-u| \leq 2$ as well as $\{\orig \lcyc{\leq 4} u\}_0$. Before going into the different cases, we note that there are $\tfrac 12\Omega(\Omega-2)$ choices for $u=v_1+v_2$ (where $|v_i|=1$), and on every choice, $\{v_1,v_2\} \subseteq \omega_0$ need to  hold for a relevant contribution to arise. Taking all this into consideration, the contribution to $\widehat\Pi^{(1)}$ becomes
	\eqq{ \tfrac 12 \Omega(\Omega-2) p_c^3 \sum_{x\in\Zd} \mathds 1_{\{|x| \leq 3, |u-x|\leq 2\}} \p \Big( \{x\in\thinn{\C_0}\}_0 \cap E'(u,x;\C_0)_1
						 \mid \{v_1,v_2\}\subseteq \omega_0 \Big),  \label{eq:exp:Pi1_u2_general}}
where $v_1$ and $v_2$ is a pair of arbitrary but fixed independent unit vectors (and $u=v_1+v_2$).

\underline{The case of $|u|=2, x=\orig$ contributes $\omtwo$:} As $|u-x|=2$, the contribution to~\eqref{eq:exp:Pi1_u2_general} is at most $\Omega^2 p_c^3 \tau^{(2)}(x-u) = \omtwo$.

\underline{The case of $|u|=2, |x|=1$ contributes $\Omega^{-1} + \omtwo$:} Note that we only need to consider $x \in\{v_1,v_2\}$ (otherwise $|u-x| =3$). For these choices of $x$, both $x\in\thinn{\C_0}$ and $E'(u,x;\C_0)_1$ hold and the contribution to~\eqref{eq:exp:Pi1_u2_general} is as claimed.

\underline{The case of $|u|=2, |x|=2$ contributes $\Omega^{-1} + \omtwo$:} By the indicator in~\eqref{eq:exp:Pi1_u2_general}, we only consider $|x-u|=2$. Let first $\|x\|_\infty=2$. There are only two such points at distance $2$ of $u$, and so the contribution to~\eqref{eq:exp:Pi1_u2_general} is at most $\Omega^2 p_c^3 \tau^{(2)}(x-u) = \omtwo$.

Let thus $x$ be one of the $2(\Omega-3)$ points with $\|x\|_\infty=1$. W.l.o.g., we assume that $x=v_1+v_3$, where $|v_3|=1$. If $v_3=-v_2$, then the contribution is bounded by $\Omega^2 p_c^3 \tau^{(2)}(x-u) = \omtwo$. Let $x$ be one of the remaining $2(\Omega-4)$ points with $\dim\gen{v_1,v_2,v_3} =3$. As $x \sim v_1$, the event $x\in\thinn{\C_0}$ holds. We partition $E'(u,x;\C_0)_1$ into whether $\{u \lconn{=2} x\}_1$ or $\{u \lconn{\geq 4} x\}_1$ and see that in the latter case, the contribution to~\eqref{eq:exp:Pi1_u2_general} is at most $\Omega^3 p_c^3 \tau^{(4)}(x-u) = \omtwo$.

For the existence of a path of length $2$, either $v_1$ or $v_4:=x+v_2$ need to be $\omega_1$-occupied. As $v_1\in\C_0$, it cannot be a pivotal point for the $\omega_1$-connection between $u$ and $x$ and there needs to be another path. The contribution to~\eqref{eq:exp:Pi1_u2_general} is therefore at most $\Omega^3 p_c^4 \tau^{(2)}(x-u) = \omtwo$. We observe that
	\[E'(u,x;\C_0)_1 \cap \{v_4\in \omega_1\} = \{v_4\in\omega_1\} \cap \Big(\{v_4\notin\piv{u,x}\}_1 \cup \{\orig \centernot\longleftrightarrow v_4 \text{ in } \Zd\setminus\{u\}\}_0  \Big). \]
As previously, $\p(v_4\notin\piv{u,x}) = \O(\Omega^{-1})$ and $\p(\orig \centernot\longleftrightarrow v_4 \text{ in } \Zd\setminus\{u\}) = 1-\O(\Omega^{-1})$, and so the contribution to~\eqref{eq:exp:Pi1_u2_general} is
	\[\Omega^3(1-\O(\Omega^{-1})) p_c^4 (1+\O(\Omega^{-1})) = \Omega^{-1} + \omtwo.\]
\underline{The case of $|u|=2, |x|=3$ contributes $\Omega^{-1} + \omtwo$:} We only need to consider neighbors of $u$, otherwise $|u-x|\geq 3$. Recall that for $|u-x|=1$, the event $E'(u,x;\C_0)_1$ holds precisely when $x \in\thinn{\C_0}$. Under our conditioning, $x$ must be connected to $\{0,v_1,v_2\}$. Note that there are two choices for $x$ with $\|x\|_\infty=2$. Since $\p(x\in\thinn{\C_0}) \leq 3 \max_{y\in\{0,v_1,v_2\}} \tau^{(2)}(x-y)) = \omone$, we may focus on the $\Omega-2$ choices of $x$ with $\|x\|_\infty=1$.

Let $x=u+v_3$ and set $z_1:=v_1+v_3, z_2:=v_2+v_3$. If $\{z_1,z_2\} \cap \omega_0 = \varnothing$, then $\{\orig \lconn{5} x\}_0$ holds, and the contribution to~\eqref{eq:exp:Pi1_u2_general} is at most $\Omega^3p_c^3 \max_{y\in\{0,v_1,v_2\}} \tau^{(3)}(x-y) = \omtwo$. If $\{z_1,z_2\} \subset \omega_0$, then the contribution to~\eqref{eq:exp:Pi1_u2_general} is at most $\Omega^3 p_c^5 = \omtwo$.

We consider the case where $z_1\notin\omega_0\ni z_2$ and respect the other case with a factor of $2$. The contribution to~\eqref{eq:exp:Pi1_u2_general} is
	\[\Omega^3(1+\O(\Omega^{-1}) p_c^4 (1+\O(\Omega^{-1})) = \Omega^{-1} + \omtwo.\]
This finishes the analysis of $\widehat\Pi^{(1)}$.
\end{proof}

\subsection{Analysis of $\widehat\Pi^{(2)}$}
\begin{lemma}[Asymptotics of $\widehat\Pi^{(2)}$] \label{lem:exp:Pi2_finer}
As $d\to\infty$,
	\[ \widehat\Pi^{(2)} =  10 \Omega^{-1} + \omtwo.\]
\end{lemma}
\begin{proof}
For the proof, we recall that
	\eqq{ \widehat\Pi^{(2)} = p_c^2 \sum_{u,v,x\in\Zd} \p\Big( \{\orig \Longleftrightarrow u\}_0 \cap E'(u,v;\C_0)_1 \cap E'(v,x;\C_1)_2 \Big), \label{eq:exp:pitwo_def}}
where $\C_0= \widetilde{\C}^{u}(\orig;\omega_0)$ and $\C_1=\widetilde{\C}^v(u;\omega_1)$. We first show that when either $v\notin\thinn{\C_0}$ or $x\notin\thinn{\C_1}$, then the contribution to $\widehat{\Pi}^{(2)}$ is $\omtwo$. Indeed, by Lemma~\ref{lem:prelim:diagrammatic_bound_Pi2} and Proposition~\ref{thm:prelim:triangle_bounds},
	\begin{align}p_c^2 \sum_{u,v,x\in\Zd} \p & \Big( \{\orig \Longleftrightarrow u\}_0 \cap E'(u,v;\C_0)_1 \cap E'(v,x;\C_1)_2 \cap\big(\{v\notin\thinn{\C_0}\}_0 \cup \{x\notin\thinn{\C_1}\}_1 \big) \Big)  \notag\\
		& \leq p_c^2 \sum \Big( p_c^3 \mathrel{\raisebox{-0.25 cm}{\includegraphics{N2_exp_11.pdf}}} 
				+ p_c^2 \mathrel{\raisebox{-0.25 cm}{\includegraphics{N2_exp_12.pdf}}} \notag\\
		& \qquad + p_c^2 \mathrel{\raisebox{-0.25 cm}{\includegraphics{N2_exp_21.pdf}}}
				+ p_c \mathrel{\raisebox{-0.25 cm}{\includegraphics{N2_exp_31.pdf}}}
				+ p_c \mathrel{\raisebox{-0.25 cm}{\includegraphics{N2_exp_22.pdf}}} \Big) \label{eq:exp:pitwo_xv_bound_A} \\
		& \leq \sum \mathrel{\raisebox{-0.25 cm}{\includegraphics{N1_exp_A_bound1.pdf}}} 
					\Big(\triangle_{p_c}^{\bullet} (\triangle_{p_c}^{\bullet\circ})^2 \triangle_{p_c} + \triangle_{p_c}^{\bullet} \triangle_{p_c} \triangle_{p_c}^{\bullet \circ}
								 + \triangle_{p_c}^{\bullet} \triangle_{p_c}^{\bullet\circ} \triangle_{p_c}^{\bullet\bullet\circ} \triangle_{p_c}
								 + \triangle_{p_c}^{\bullet}(\triangle_{p_c}^{\bullet\circ})^2 \Big) \notag \\
		& \qquad + p_c^3 \sum \Big(p_c \mathrel{\raisebox{-0.25 cm}{\includegraphics{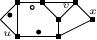}}}
				+ \mathrel{\raisebox{-0.25 cm}{\includegraphics{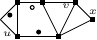}}} \Big) \label{eq:exp:pitwo_xv_bound_B} \\
		& \leq  \O(\Omega^{-3}) \sum \mathrel{\raisebox{-0.25 cm}{\includegraphics{N1_exp_A_bound1.pdf}}} \ = \omtwo.  \notag \end{align}
We expanded the third diagram in~\eqref{eq:exp:pitwo_xv_bound_A} to get the two diagrams of~\eqref{eq:exp:pitwo_xv_bound_B}. We next show that only $|u|=1$ gives a relevant contribution. Indeed,
	\al{p_c^2 \sum_{u,v,x\in\Zd: |u|\geq 2} \p & \Big( \{\orig \Longleftrightarrow u\}_0 \cap E'(u,v;\C_0)_1 \cap E'(v,x;\C_1)_2 \cap \{v\in\thinn{\C_0}\}_0 \cap \{x\in\thinn{\C_1}\}_1 \Big) \\
			& \leq p_c^2 \sum \mathds 1_{\{|u| \geq 2\}} \mathrel{\raisebox{-0.25 cm}{\includegraphics{N2_exp_32.pdf}}}
				\ \leq \triangle_{p_c}^{\bullet}\triangle_{p_c}^{\bullet\circ} \sum \mathds 1_{\{|u| \geq 2\}} \mathrel{\raisebox{-0.25 cm}{\includegraphics{N1_exp_A_bound1.pdf}}} \ = \omtwo.}
We can thus fix $u$ to be an arbitrary neighbor of the origin and need to investigate
	\eqq{ \Omega p_c^2 \sum_{v,x\in\Zd} \p \Big( E'(u,v;\C_0)_1 \cap E'(v,x;\C_1)_2 \cap \{v\in\thinn{\C_0}\}_0 \cap \{x\in\thinn{\C_1}\}_1 \Big). \label{eq:exp:pitwo_reduced_def}}
Before going into specific cases, we exclude some of them right away: When $|x| \vee |u-x| \geq 4$, then the contribution to~\eqref{eq:exp:pitwo_reduced_def} is
	\[ p_c^2 \sum \mathds 1_{\{|x| \vee |u-x| \geq 4\}} \mathrel{\raisebox{-0.25 cm}{\includegraphics{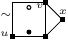}}} \ \leq \sum_{u,t,v,x} \dtrial^{(9)}(u,t,v,x) = \omtwo\]
by Lemma~\ref{lem:exp:l_step_connections}. In the above, a line decorated with a `$\sim$' symbol denotes a direct edge. Similarly, when $|v| \geq 3$ or $|x-v|\geq 3$, the contribution to~\eqref{eq:exp:pitwo_reduced_def} is at most
	\al{ p_c^2 \sum & \mathds 1_{\{|v| \vee |x-v| \geq 3\}} \mathrel{\raisebox{-0.25 cm}{\includegraphics{N2_exp_32_u1.pdf}}}
		 \ \leq  p_c \triangle_{p_c}^\bullet \big(\tau^{(3)}\ast\tau\ast\tau^{\bullet}\ast\jeq \big)(\orig) +
		 			  p_c^2 \sum \mathds 1_{\{|x-v| \geq 3\}} \mathrel{\raisebox{-0.25 cm}{\includegraphics{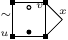}}} \\
		 & \leq p_c \triangle_{p_c}^\bullet \sum_{u,v} \triangle^{(6)}(\orig,u,v) + p_c^4 \big(\jeq^{\ast 3} \ast \tau^{\ast 3}\big)(\orig) 
		 				+ p_c \triangle_{p_c}^{\bullet\circ} \sum_{t,x} \triangle^{(6)}(\orig,t,x) = \omtwo. }
We now investigate~\eqref{eq:exp:pitwo_reduced_def} by splitting the double sum over $v$ and $x$. We organize this by considering the three main cases for $|v|\in\{0,1,2\}$. An overview of the contributions is given in the following table:
\begin{center}
\begin{tabular}{ c||c|c|c|c } 
 $\widehat{\Pi}^{(2)}$:  & $x=\orig$ & $|x|=1$ & $|x|=2$ & $|x|=3$ \\ 
 \hline \hline
 \rule{0pt}{3ex} $v=\orig$ &  & $2 \Omega^{-1}$ & $\Omega^{-1}$ & \\
 $|v|=1$ & $\Omega^{-1}$ &  & $2\Omega^{-1}$ & $\Omega^{-1}$ \\ 
 $|v|=2$ &  & $\Omega^{-1}$ & $\Omega^{-1}$ & $\Omega^{-1}$
\end{tabular}
\end{center}
\paragraph{Contributions of $v=\orig$.} The events $E'(u,v;\C_0)_1$ and $\{v\in \thinn{\C_0}\}$ hold.

\underline{The case of $|x|=1$ contributes $2\Omega^{-1} + \omtwo$:} First, consider the choice of $x=u$. It is easy to see that the event in~\eqref{eq:exp:pitwo_reduced_def} holds and the contribution is $\Omega p_c^2 = \Omega^{-1} + \omtwo$.

Consider $0\sim x\neq u$. As $v\sim x$, we have $E'(v,x;\C_1)_2 = \{x \in \thinn{\C_1}\}_1$. If $x=-u$, then $\{x \in \thinn{\C_1}\}_1 \subseteq \{u \lconn{4} x\}_1$ and we receive a contribution of order $\omtwo$. Consider now one of the $\Omega-2$ remaining choices for $x$ and set $z=u+x$. Then
	\[ \p(x \in \thinn{\C_1}) = \p(z \in\omega_1) + \p(z \notin\omega_1, x \in \thinn{\C_1}) = p_c + \O(\tau^{(4)}(x-u)) = p_c + \omtwo,\]
yielding a contribution to~\eqref{eq:exp:pitwo_reduced_def} of $\Omega(\Omega-2)p_c^3 + \omtwo = \Omega^{-1}+\omtwo$.

\underline{The case of $|x|=2$ contributes $\Omega^{-1} + \omtwo$:} 
If $|u-x|=3$, then the contribution to~\eqref{eq:exp:pitwo_reduced_def} is bounded by $\Omega^3 p_c^2 \tau^{(2)}(x-v) \tau^{(3)}(u-x) =\omtwo$. Similarly, if $x=2u$, we obtain a bound of $\Omega p_c^2 \tau^{(2)}(x-v) = \omtwo$. Let therefore $x$ be one of the $\Omega-2$ remaining neighbors of $u$ and note that $\{x\in\thinn{\C_1}\}$ holds.

We set $z=x-u$. If $z\notin\omega_2$, then $E'(v,x;\C_1)_2 \subseteq \{ v \lconn{4} x\}_2$ by Observation~\ref{obs:exp:pivotality}, and the contribution to~\eqref{eq:exp:pitwo_reduced_def} is at most $\Omega^2 p_c^2 \tau^{(4)}(x-v) = \omtwo$. If $z\in\omega_2$, then $E'(v,x;\C_1)_2 = \{z \notin\piv{v,x}\}_2 \cup \{z \notin\thinn{\C_1}  \}_1$. By a similar argument to the one below~\eqref{eq:exp:Pi1_u1_x1_gamma1_path}, the contribution to~\eqref{eq:exp:pitwo_reduced_def} becomes
	\[ \Omega(\Omega-2) p_c^2 \p \Big(\{z\in \omega_2 \} \cap \big(\{z \notin\piv{v,x}\}_2 \cup \{z \notin\thinn{\C_1}  \}_1\big)\Big) = \Omega^2 p_c^3 (1-\omone) = \Omega^{-1} + \omtwo. \]  

\underline{The case of $|x|=3$ contributes $\omtwo$:} Distinguishing between $|u-x|=4$ (at most $\Omega^3$ choices for $x$) and $|u-x|=2$ (at most $\Omega^2$ choices), the contribution to~\eqref{eq:exp:pitwo_reduced_def} is at most
	\[\Omega p_c^2 \tau^{(3)}(x-v) \big( \Omega^3 \tau^{(4)}(u-x) + \Omega^2 \tau^{(2)}(u-x) \big) = \omtwo.\]

\paragraph{Contributions of $|v|=1$.} Let us first consider $v=-u$ and show that this case contributes $\omtwo$. Indeed, $E'(u,v;\C_0)_1 \subseteq \{u \lconn{4} v\}_1$ by Observation~\ref{obs:exp:pivotality}. With the further inclusion $E'(v,x;\C_1)_2 \cap \{x \in\thinn{\C_1}\} \subseteq \{v \longleftrightarrow x\}_2$, we have that the contribution to~\eqref{eq:exp:pitwo_reduced_def} is at most
	\al{ \Omega p_c^2 \tau^{(4)} &(u-v) \Big( \sum_{|x|=1} \tau^{(2)}(x-v) + \sum_{x: v\sim x} 1 + \sum_{|x|=2,|x-v|=3} \tau^{(3)}(x-v) \\
			& \qquad + \sum_{|x|=3,|x-v|=2} \tau^{(2)}(x-v) +\sum_{|x|=3,|x-v|=4} \tau^{(4)}(x-v) \Big)  \\
			& \leq \O(\Omega^{-3}) \Big(\Omega \O(\Omega^{-1}) + \Omega + \Omega^2 \omtwo + \Omega^2 \O(\Omega^{-1}) + \Omega^3 \O(\Omega^{-3}) \Big) = \omtwo.}
We may therefore take $v\neq \pm u$ to be one of the $\Omega-2$ remaining neighbors of the origin. Set $t=v+u$. We first claim that $t \notin \omega_1$ results in an $\omtwo$ contribution. Note that, by Observation~\ref{obs:exp:pivotality}, $E'(u,v;\C_0)_1 \cap \{t \notin \omega_1\} \subseteq \{u \lconn{4} v\}_1$. As there is only one choice of $x$ such that $u \sim x \sim v$ and at most $\Omega$ choices such that $|x|=3$ and $x \sim v$, we can bound ~\eqref{eq:exp:pitwo_reduced_def} by
	\al{ \Omega^2 p_c^2 \sum_{x\in\Zd} & \Big( \tau^{(4)}(v-u) \big(\mathds 1_{\{x=\orig\}} + \mathds 1_{\{|x|=1\}} \tau^{(2)}(x-v) + \mathds 1_{\{|x|=2, u\sim x \sim v\}}  \\
			& \quad + \mathds 1_{\{|x|=3, |u-x|=2=|v-x|\}} \tau^{(2)}(x-v)  \big) + \mathds 1_{|x|=2, v \sim x} \p \big(E'(u,v;\C_0)_1 \cap \{ t \notin\omega_1 \}  \cap \{x \in \thinn{\C_1}\} \big) \Big) \\
			\leq & \omtwo \big( 2 + 2\Omega \tau^{(2)}(x-v)\big) + \mathcal O(1) \sum_{|x|=2, x \sim v} \p^{(2)} \big(E'(u,v;\C_0)_1 \cap \{ t \notin\omega_1 \} \cap \{x \in \thinn{\C_1}\} \big). }
It remains to bound the last probability. There are at most $\Omega$ choices for $x$. If $\{u \lconn{5} x\}$, then the contribution is $\omtwo$. Note that the $u$-$v$-path in $\omega_1$ cannot use and is independent of the status of $\orig$, as the origin may not be a pivotal point. Hence, if $\orig \in\omega_1$, the contribution is at most $\Omega p_c \tau^{(4)}(v-u) = \omtwo$. We therefore assume $\orig\notin\omega_1$ and aim to bound
	\eqq{ \Omega \p \big(\{u \lconn{4} v\}_1  \cap \{ \orig,t\notin \omega_1 \} \cap \{ u \lconn{\leq 3} x \}_1  \big)  \label{eq:exp:pitwo_v1_t_vacant}}
When avoiding $\orig$ and $t$, there are only two $u$-$x$-paths of length $3$, namely $\gamma_1=(u,y,z,x)$ and $\gamma_2=(u,y,y-u,x)$, where $y:=x+u-v$ and $z:=y+v$. See Figure~\ref{fig:exp:v1_a} for an illustration. But now,~\eqref{eq:exp:pitwo_v1_t_vacant} is bounded by
	\al{ \Omega \p \Big( & \{\orig,t\notin \omega_1\} \cap \bigcup_{i=1,2} \bigcup_{s \in \gamma_i \setminus\{x\} } \{\gamma_i \subseteq \omega_1\} \circ \{s \longleftrightarrow v\}_1 \Big) \\
		& \leq 2 \Omega p_c^2 \big( \tau^{(4)}(v-u) + \tau^{(3)}(y-v) + 2\tau^{(2)}(z-v) \big) = \omtwo. }
As a consequence, we can focus on $t\in\omega_1$, and~\eqref{eq:exp:pitwo_reduced_def} reduces to
	\[ \Omega(\Omega-2) p_c^2 \sum_{x\in\Zd} \p \Big( E'(u,v;\C_0)_1 \cap E'(v,x;\C_1)_2 \cap \{t\in\omega_1, x\in\thinn{\C_1}\}_1 \Big). \]
But under $t \in\omega_1$, we have $E'(u,v;\C_0)_1 = \{t \notin\piv{u,v}\}_1 \cup \{t \notin \thinn{\C_0}\}_0$. The latter event has probability $1-\omone$, and so we can can instead investigate
	\eqq{ \Omega^2 p_c^2 (1-\omone) \sum_{x\in\Zd} \p \Big( E'(v,x;\C_1)_2 \cap \{t\in\omega_1, x\in\thinn{\C_1}\}_1 \Big), \label{eq:exp:pitwo_reduced_v1}}
where $u$ and $v$ are two arbitrary (but fixed) neighbors of $\orig$ (satisfying ($u \neq \pm v$).

\underline{The contribution of $x=\orig$ is $\Omega^{-1} + \omtwo$:} Note that $x \in \thinn{\C_1}$ holds, and so does $E'(v,x;\C_1)_2$. Hence, the contribution to~\eqref{eq:exp:pitwo_reduced_v1} is $\Omega^{-1} + \omtwo$.

\begin{figure}
     \centering
     \begin{subfigure}[b]{0.3\textwidth}
         \centering
         \includegraphics[scale=1.5]{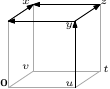}
	\caption{The case $|x|=2, t \notin \omega_1$.}
         \label{fig:exp:v1_a}
     \end{subfigure}
     \hfill
     \begin{subfigure}[b]{0.3\textwidth}
         \centering
         \includegraphics[scale=1.5]{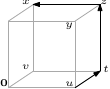}
	\caption{The case $|x|=2, u\nsim x$.}
         \label{fig:exp:v1_b}
	\end{subfigure}     
     \hfill
     \begin{subfigure}[b]{0.3\textwidth}
         \centering
         \includegraphics[scale=1.5]{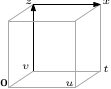}
	\caption{The case $|x|=3$.}
         \label{fig:exp:v1_c} 
     \end{subfigure}
        \caption{An illustration of several appearing cases for $|v|=1$. In (a), the two paths from $u$ to $x$ of length $3$ that avoid $\orig$ and $t$ are drawn. In (b), the path along $t,z$ which ensures $x \in \thinn{\C_1}$ for a contribution of $\Omega^{-1}$ is drawn. In (c), the scenario $|x-u|=2=|v-x|$ is shown, and the path along $z$ ensuring $\{v \longleftrightarrow x\}_2$ is drawn in black.}
        \label{fig:exp:v1}
\end{figure}

\underline{The contribution of $|x|=1$ is $\omtwo$:} If $x \in\{\pm u, -v\}$, we can bound the contribution to~\eqref{eq:exp:pitwo_reduced_v1} by $\Omega^2 p_c^2 \tau^{(2)}(u-v)\tau^{(2)}(x-v) = \omtwo$ (as both $\{v \longleftrightarrow x\}_2$ and $\{u\longleftrightarrow v\}_1$ need to hold). Consider thus one of the $\Omega-4$ choices for $x$ satisfying $\dim\gen{u,v,x}=3$. Conditional on $t \in \omega_1$, we have $\{x\in\thinn{\C_1}\}_1 \subseteq \{u \lconn{2} x\}_1 \cup \{t \lconn{3} x\}_1$, and so the contribution is at most
	\[\Omega^3 p_c^3 \taup^{(2)}(x-v) \big( \taup^{(2)}(x-u) + \taup^{(3)}(x-t) \big) = \omtwo.\]

\underline{The contribution of $|x|=2$ is $2\Omega^{-1}+\omtwo$:} We can restrict to the choices of $x$ where $v\sim x$ by the considerations made in the beginning of the proof.
\begin{compactitem}
\item Let $x \sim u$. There is only one choice for $x$ such that $|u-x|=|v-x|=1$, namely $x=t$. For this choice, $E'(v,x;\C_1)_2$ certainly holds, and also $x\in\thinn{\C_1}$. We get a contribution of $\Omega^{-1}+\omtwo$.
\item Let $x \not\sim u$. There are $\Omega-2$ choices for $x$. We first exclude $x=v-u$. As $\p(x \in \thinn{\C_1} \mid t\in\omega_1) \leq \tau^{(4)}(x-t) + \tau^{(3)}(x-u) = \omtwo$, the contribution in total is $\omtwo$.

Let now $x$ be one of the $\Omega-3$ remaining neighbors of $v$. As $v\sim x$, we have $E'(v,x;\C_1)_2 = \{x \in \thinn{\C_1}\}$. We set $z=x+u$ (see Figure~\ref{fig:exp:v1_b}) and assume first that $z \notin\omega_1$. Then
	\[ \{z\notin\omega_1 \ni t, x\in\thinn{\C_1}\} \subseteq \{z\notin\omega_1 \ni t\} \cap \big(\{u \lconn{3} x \text{ off } \{t\} \cup \{t \lconn{4} x \} \big) \]
and the contribution to~\eqref{eq:exp:pitwo_reduced_v1} is at most $\Omega^2 p_c^3 (1-\omone) (\tau^{(3)}(x-u) + \tau^{(4)}(x-t)) = \omtwo$. On the other hand, if $z\in\omega_1$, then $x \in \thinn{\C_1}$ holds and~\eqref{eq:exp:pitwo_reduced_v1} becomes
	\[ \Omega^2 p_c^2 (1-\omone) (\Omega-3) \p(t,z \in \omega_1) = \Omega^{-1} + \omtwo.\]
\end{compactitem}

\underline{The contribution of $|x|=3$ is $\Omega^{-1} + \omtwo$:} There are at most $\Omega^3$ choices for $x$ such that $|u-x|=|v-x|=4$ and there are at most $2\Omega^2$ choices where $|x-u|\neq |v-x|$. The contribution of those $x$ to~\eqref{eq:exp:pitwo_reduced_v1} is therefore bounded by
	\[\Omega^2 p_c^2  \Big(\sum_{|x|=3} \tau^{(4)}(x-v) \tau^{(4)}(u-x) + 2 \sum_{|x-v|=2\neq |u-x|} \tau^{(2)}(x-v) \tau^{(4)}(u-x) \Big) = \omtwo. \]
It remains to investigate those $x$ with $|u-x| = 2 = |v-x|$. This is only possible when $x\sim t$. Let first $x = 2u+v$. By Observation~\ref{obs:exp:pivotality}, $E'(v,x;\C_1)_2 \cap \{t\in\omega_1\} \subseteq \{v \lconn{4} x\}$, and~\eqref{eq:exp:pitwo_reduced_v1} is at most $\Omega^2 p_c^2 \tau^{(4)}(x-v) = \omtwo$.

Let now $x$ be one of the $\Omega-3$ remaining neighbors of $t$ (note that either $\|x\|_\infty=1$ or $x=2v+u$). We set $z:=x-u$ and point to Figure~\ref{fig:exp:v1_c} for an illustration. As $t$ is occupied in $\omega_1$, we have $x \in\thinn{\C_1}$. Assume now $z\notin\omega_2$. By Observation~\ref{obs:exp:pivotality}, $E'(v,x;\C_1)_2 \subseteq \{ v \lconn{4} x\}_2$ and the contribution to~\eqref{eq:exp:pitwo_reduced_v1} is at most $\Omega^2 p_c^3 \tau^{(4)}(x-v) = \omtwo$. On the other hand, if $z\in\omega_2$,~\eqref{eq:exp:pitwo_reduced_v1} becomes
	\[ (1+\omone) (\Omega-3) \p \big( \{t\in\omega_1, z \in\omega_2\} \cap \big(\{z \notin \thinn{\C_1}\}_1 \cup \{z \notin\piv{v,x}\}_2 \big) \big) = \Omega^{-1} + \omtwo. \]
Again, we have used that $\{z \notin \thinn{\C_1}\}_1$ has probability $1-\omone$ conditional on $t\in\omega_1$.

\paragraph{Contributions of $|v|=2$.} We first show that when $|u-v| =3$, no relevant contributions arise. Indeed, for those $v$,~\eqref{eq:exp:pitwo_def} is at most
	\al{p_c^2 \sum & \mathds 1_{\{|v|=2, |u-v|=3\}} \mathrel{\raisebox{-0.25 cm}{\includegraphics{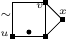}}} 
			\ \leq p_c \sum \Big( \mathds 1_{\{|v|=2, |u-v|=3\}} \mathrel{\raisebox{-0.25 cm}{\includegraphics{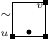}}}
				\Big( \sup_{\textcolor{darkorange}{\bullet}, \textcolor{blue}{\bullet}} 
					p_c \sum \mathrel{\raisebox{-0.25 cm}{\includegraphics{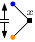}}} \Big)\Big) \\
		& \leq \triangle^\bullet_{p_c} \Big( \sum \mathds 1_{\{|v|=2, |u-v|=3\}} p_c \mathrel{\raisebox{-0.25 cm}{\includegraphics{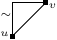}}}
				+ p_c^2 \sum \mathds 1_{\{|v|=2, |u-v|=3\}} \mathrel{\raisebox{-0.25 cm}{\includegraphics{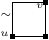}}} \Big) \\
		& \leq \triangle^\bullet_{p_c} \Big( p_c \sum_{u,v} \triangle^{(6)}(u,v,\orig) + 2 p_c^4 \big(\jeq^{\ast 3}\ast\tau^{\ast 3}\big)(\orig) \Big) = \omtwo.  }
Moreover, $v=2u$ implies $\{v\in\thinn{\C_0}\}_0 \subseteq \{\orig \lconn{4} v\}_0$. We can thus bound the contribution to~\eqref{eq:exp:pitwo_reduced_def} by $\Omega p_c \tripf \tau^{(4)}(v) = \omtwo$. Let $v$ be one of the $\Omega-2$ remaining neighbors of $u$, implying $E'(v,u;\C_0)=\{v\in\thinn{\C_0}\}_0$. Let $z=v-u$. Then for $v \in\thinn{\C_0}$ to hold, either $z \in\omega_0$ or there must be a path of length at least $4$. In the latter case, we can bound~\eqref{eq:exp:pitwo_reduced_def} by $p_c^2 \sum_{u,v,t,x} \dtrial^{(9)}(u,t,v,x) = \omtwo$. We can therefore restrict to investigating
	\eqq{ \Omega(\Omega-2) p_c^3 \sum_{x\in\Zd} \p \Big( E'(v,x;\C_1)_2 \cap \{x\in\thinn{\C_1}\}_1 \Big), \label{eq:exp:pitwo_reduced_v2}}
where $u$ is an arbitrary (but fixed) neighbor of $\orig$ and $v \notin\{\orig, 2u\}$ is some fixed neighbor of $u$.

\underline{The contribution of $x=\orig$ is $\omtwo$:} As $\{\orig \longleftrightarrow v\}_2$ needs to hold, we get a bound on~\eqref{eq:exp:pitwo_reduced_v2} by $\Omega^2 p_c^3 \tau^{(2)}(v) = \omtwo$.

\underline{The contribution of $|x|=1$ is $\Omega^{-1}+\omtwo$:} We only need to consider $|v-x|=1$, and there are two such choices for $x$. If $x=v-u$, then the contribution is bounded by $\Omega^2p_c^3 \tau^{(2)}(u-x) = \omtwo$.

On the other hand, if $x=u$, both $E'(v,x;\C_1)_2$ and $\{x \in\thinn{\C_1}\}_1$ hold and the contribution to~\eqref{eq:exp:pitwo_reduced_v2} is $\Omega^{-1} + \omtwo$.

\underline{The contribution of $|x|=2$ is $\Omega^{-1}+\omtwo$:} Note that only $|v-x|=2$ may produce relevant contributions. Writing $v=u+z$, we first consider $x=u-z$. Again, $E'(v,x;\C_1)_2 \subseteq \{v \lconn{4} x\}_2$ by Observation~\ref{obs:exp:pivotality}, and so the contribution to~\eqref{eq:exp:pitwo_reduced_v2} is at most $\Omega^2 p_c^2 \tau^{(4)}(v-x) = \omtwo$. Similarly, If $|u-x|=3$, the contribution is at most $\Omega^3 p_c^3 \tau^{(2)}(v-x) \tau^{(3)}(u-x) = \omtwo$.

Let now $y$ be one of the $\Omega-4$ unit vectors satisfying $\dim\gen{u,z,y}=3$. Write $x=u+y$ and set $t=x+z=v+y$. We claim that we only get a relevant contribution if $t\in\omega_2$: As $\{t\notin\omega_2\} \subseteq \{v \lconn{4} x\}_2$ by Observation~\ref{obs:exp:pivotality}, this gives a bound on the contribution to~\eqref{eq:exp:pitwo_reduced_v2} by $\Omega^3 p_c^3 \tau^{(4)}(x-v) = \omtwo$. Under $t\in\omega_2$,~\eqref{eq:exp:pitwo_reduced_v2} becomes
	\aln{ \Omega^3(1-\omone) p_c^3 &\p \Big(\{t\in\omega_2\} \cap \big(\{t \notin\piv{v,x}\}_2 \cup \{t \notin\thinn{\C_1}\}_1 \big) \Big) \\
		 & = \Omega^3(1-\omone) p_c^4 (1-\omone) = \Omega^{-1} + \omtwo. \label{eq:exp:pitwo_v2_x2}}

\underline{The contribution of $|x|=3$ is $\Omega^{-1}+\omtwo$:} We only need to consider terms where $|v-x|=1$. Let $x=v+y$, where $|y|=1$. If $y=z$, then $\{x\in\thinn{\C_1}\} \subseteq \{v \lconn{4} x\}_1$ and the contribution to~\eqref{eq:exp:pitwo_reduced_v2} is $\O(\Omega^{-3})$. For the other $\Omega-2$ choices for $x$, we set $t=u+y$. When $t\notin\omega_2$, we require $\{x\in\thinn{\C_1}\} \subseteq \{v \lconn{4} x\}_1$ and the contribution is $\omtwo$. When $t\in\omega_2$, the contribution is identical to~\eqref{eq:exp:pitwo_v2_x2} and hence $\Omega^{-1} + \omtwo$.
\end{proof}


\providecommand{\bysame}{\leavevmode\hbox to3em{\hrulefill}\thinspace}
\providecommand{\MR}{\relax\ifhmode\unskip\space\fi MR }
\providecommand{\MRhref}[2]{%
  \href{http://www.ams.org/mathscinet-getitem?mr=#1}{#2}
}
\providecommand{\href}[2]{#2}

\end{document}